\documentclass[onefignum,onetabnum]{siamart190516}

\usepackage{lipsum}
\usepackage{amsfonts}
\usepackage{graphicx}
\usepackage{epstopdf}
\usepackage{algorithmic}
\ifpdf
  \DeclareGraphicsExtensions{.eps,.pdf,.png,.jpg}
\else
  \DeclareGraphicsExtensions{.eps}
\fi

\newsiamremark{remark}{Remark}
\newsiamremark{hypothesis}{Hypothesis}
\crefname{hypothesis}{Hypothesis}{Hypotheses}
\newsiamthm{claim}{Claim}

\headers{Moment-driven predictive control of collective dynamics}{G. Albi, M. Herty, D. Kalise and C. Segala}

\title{Moment-driven predictive control of mean-field collective dynamics\thanks{Submitted to the editors DATE.
\funding{GA and CS thank the Italian Ministry of Instruction, University and Research (MIUR) to support this research with funds coming from PRIN Project 2017 (No. 2017KKJP4X entitled “Innovative numerical methods for evolutionary partial differential equations and applications”) and program Department of Excellence. GA and CS are member of the {\em INdAM-GNCS}. MH thanks HE5386/19,18.}}}

\author{Giacomo Albi\thanks{Department of Computer Science, University of Verona, Str. Le Grazie 15, Verona, I-37134, Italy (\email{giacomo.albi@univr.it})} \and Michael Herty\thanks{IGPM, RWTH Aachen University, Templergraben, 55, D-52062 Aachen, Germany (\email{herty@igpm.rtwh-aachen.de})} \and Dante Kalise\thanks{School of Mathematical Sciences, University of Nottingham, University Park, Nottingham NG7 2QL, United Kingdom (\email{dante.kalise@nottingham.ac.uk})}\and Chiara Segala\thanks{Department of Mathematics, University of Trento, Sommarive , Trento, Italy (\email{chiara.segala-1@unitn.it})}
}

\usepackage{amsopn}

\usepackage{graphicx}
\usepackage{epstopdf}
\usepackage{rotating}
\usepackage{latexsym}
\usepackage{mathrsfs}
\usepackage{hyperref}
\usepackage{amscd}
\usepackage{sidecap}
\usepackage{color}
\usepackage{algorithm}
\usepackage{algorithmic} 
\usepackage{tikz}
\usepackage{enumerate}
\usepackage{array}
\usepackage{amsmath}
\usepackage{booktabs}

\numberwithin{equation}{section}
\newcolumntype{P}[1]{>{\centering\arraybackslash}p{#1}}
\newcolumntype{M}[1]{>{\centering\arraybackslash}m{#1}}

\graphicspath{{pics/}}

\newcommand{\R}{\mathbb R}

\newcommand{\ko}{k_o}
\newcommand{\kd}{k_d}

\def\be#1\ee{\begin{equation}#1\end{equation}}

\newtheorem{thm}{Theorem}[section] 
 
\newtheorem{lem}[thm]{Lemma} 
\newtheorem{prop}[thm]{Proposition}

\renewcommand{\theequation}{\arabic{section}.\arabic{equation}}
\renewcommand{\thetable}{\arabic{section}.\arabic{table}}
\renewcommand{\thefigure}{\arabic{section}.\arabic{figure}}
\newcommand{\bq}{\begin{equation}}
\newcommand{\eq}{\end{equation}}

\ifpdf
\hypersetup{
  pdftitle={Moments driven predictive control (MdPC) for mean-field collective dynamics},
  pdfauthor={G. Albi, M. Herty, D. Kalise and C. Segala}
}
\fi

\begin{document}
	\maketitle
	
	\begin{abstract}
		The synthesis of control laws for interacting agent-based dynamics and their mean-field limit is studied. A linearization-based approach is used for the computation of sub-optimal feedback laws obtained from the solution of differential matrix Riccati equations. Quantification of dynamic performance of such control laws leads to theoretical estimates on suitable linearization points of the nonlinear dynamics. Subsequently, the feedback laws are embedded into nonlinear model predictive control framework where the control is updated adaptively in time according to dynamic information on moments of linear mean-field dynamics. The performance and robustness of the proposed methodology is assessed through different numerical experiments in collective dynamics.
	\end{abstract}



	\tableofcontents

	\section{Introduction}
	The study of collective behaviour phenomena from a multiscale modelling perspective has seen an increased level of activity over the last years. Classical examples in socio-economy, biology and robotics are given by self-propelled particles, such animals and robots, see e.g. \cite{bellomo20review, MR2974186, MR2165531, MR3119732, MR2861587, MR2580958,Giselle}. Those particles interact according to a nonlinear model encoding various social rules as for example attraction, repulsion and alignment. A particular feature of such models is their rich dynamical structure, which include different types of emerging patterns, including consensus, flocking, and milling \cite{MR2887663, MR2247927,cucker2007emergent,d2006self,motsch2014heterophilious}.
	Understanding the impact of control inputs in such complex systems is of great relevance for applications. Results in this direction allow to design optimized actions such as collision-avoidance protocols for swarm robotics \cite{CKPP19,KPAsurvey15,MR3157726,Meurer}, pedestrian evacuation in crowd dynamics \cite{MR3542027,MR3308728,dyer2009leadership,MR4046175}, supply chain policies \cite{MR2844776,degond2007network}, the quantification of interventions in traffic management \cite{MR3948232,han2017resolving,stern2018dissipation} or in opinion dynamics \cite{AHP15,MR3268062,Garnier}.
	Here, we are concerned with  the control of high-dimensional nonlinear systems of interacting particles which can describe self-organization patterns. We will consider dynamics accounting the evolution of $N$ agents with state $v_i(t)\in\mathbb{R}^d$, undergoing a binary exchange of information weighted by a kernel $P:\R^d\times\R^d\to \R$ and forced by a control signal $u_i(t)\in\R^n$, represented as
	\begin{align}\label{nonlin_dynamics}
		\dot{v}_i &= \frac{1}{N}\sum_{j=1}^N P(v_i,v_j)(v_j-v_i) + u_i\,, \qquad v_i(0)=v_{i}^0\,,\qquad
		i=1,\ldots,N.
	\end{align}
	 While the original formulation the interacting particle system \eqref{nonlin_dynamics} is at microscopic level through a system of ODEs, the study of large particle limit has in many cases allowed to analyse emerging patterns or identifying relevant parameters. The derivation of a model hierarchy starting from dynamical systems to kinetic equations and fluid dynamic models has been studied intensively in the literature, for example in \cite{MR2536247, MR2438213, MR2740099,  MR2425606, MR3194652,carrillo2014derivation}. Of particular interest for control design purposes is the study of mean-field control approaches where the control law obtain formal independence on the number of interacting agents \cite{MR3264236,MR4028474,FPR14,MFG}. 
	The construction of computational methods for mean-field optimal control is a challenging problem due to the nonlocality and nonlinearity arising from the interaction kernel \cite{ACFK17,Pearson,liu2020computational}. Furthermore, depending  on the associated cost, non-smooth and/or non-convex optimization problems might also arise \cite{CFPT15,BAILO20181,ALBI201886}.

	In order to circumvent these difficulties we propose an approach where we  synthesize  sub-optimal feedback-type controls through the linearization of the interaction kernel and by solving the resulting linear-quadratic optimal control problem through a Riccati equation  -- all based on the  corresponding mean-field equations, similarly as in \cite{herty2015mean,herty2018suboptimal}. This approach also avoids the limitations associated to the synthesis of optimal feedback laws for high-dimensional nonlinear dynamics via the Hamilton-Jacobi-Bellman PDE \cite{dolgov2019tensor,azmi2020datadriven}.  The proposed methodology yields a control law for the linear model, which is later  embedded into the non-linear dynamics \eqref{nonlin_dynamics}. A sketch of this control concept is given in Figure \ref{f1} where we show the microscopic formulation of our approach. The main advantage of the proposed design is that unlike the classical control loop (\ref{f1}, left), we do not require a continuous measurement/estimation of the nonlinear state, nor the synthesis of a nonlinear optimal feedback law. Instead, we only require periodic measurements of the nonlinear state to update our linearized system.
	\begin{figure}[!ht]
		\begin{center}
			\includegraphics[width=\linewidth]{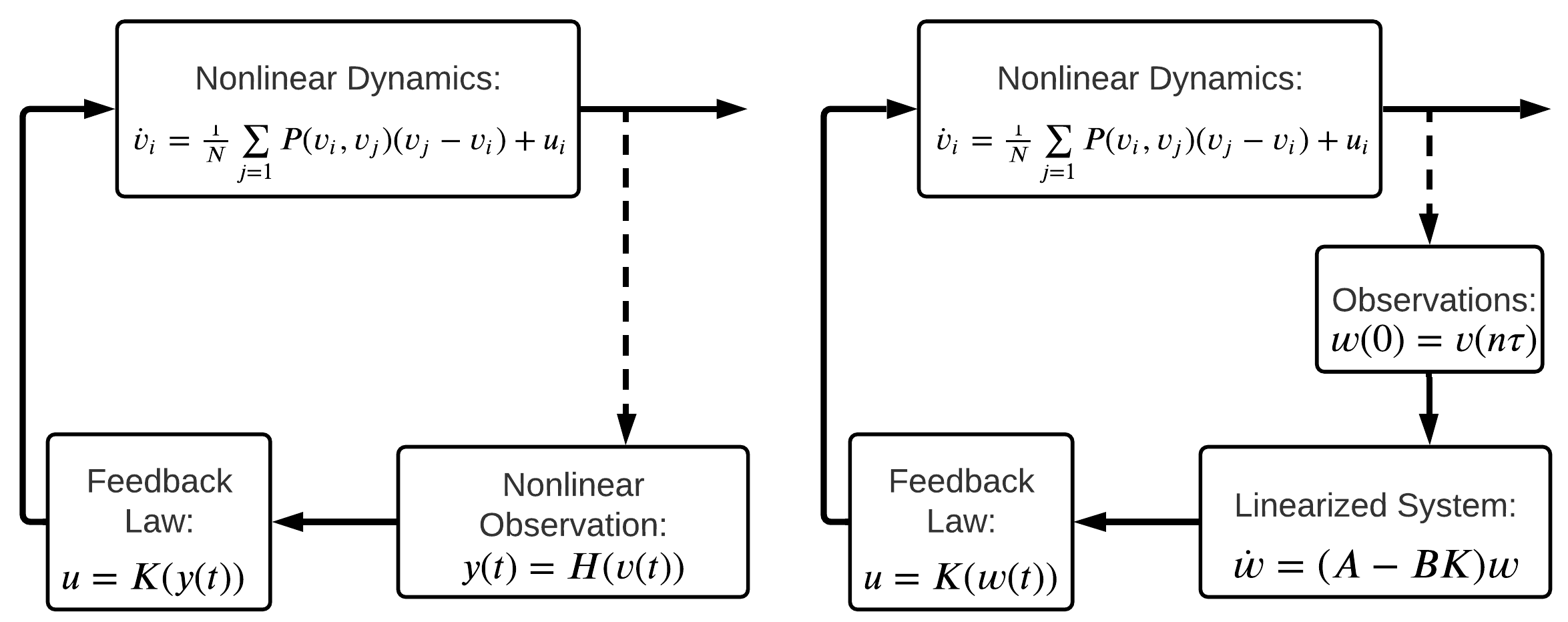}\\
			\caption{Left: the classical control loop for nonlinear dynamics. An often incomplete measurement of the state is recovered through a nonlinear observation. This observed state is inserted into a feedback law which requires the solution of a high-dimensional HJB PDE. This task is often unaffordable, and this block is replaced by a sub-optimal control law which is fed into the dynamics. Right: the Moment-driven predictive control methodology (MdPC) we propose simplifies the control loop on the left by requiring fewer measurements of the nonlinear state (every $\tau$ seconds), feeding this information into a linearized system for which the optimal feedback law can be easily computed.}
		\end{center}\label{f1}
	\end{figure} 
	
	However, using the linear optimal control within the nonlinear model does not necessarily yield a stabilizing control law, since over time the nonlinear dynamics may be far from the linearization point. Because of the latter, we aim at quantifying the impact of this control and the number of linearization updates needed to stabilize the nonlinear system. Hence, we propose different controls, distinguishing between closed-loop and open-loop strategies: in the first case the control acts having access to the full information of the non-linear system at each time. In the second case, the control only requires information available at initial time. We quantify the performances of these control approaches by estimating the decay of macroscopic quantities associated to \eqref{nonlin_dynamics} such as the first and second moments of the particle ensemble. 
	
	In order to enhance open-loop strategies we introduce a novel Moment-driven Predictive Control (MdPC) framework. Based on dynamic estimates of the moments decay, we are able to perform a forward error analysis to estimate the next  point in time where we need to update the linearization the dynamics and its feedback law. 
	This strategy can be seen as model predictive control (MPC) technique \cite{mayne2000constrained,camacho2013model,grune2017nonlinear,azmi}, where an open-loop control signal is applied only up to a subsequent point in time, after which the optimization is repeated. Moreover, the proposed control strategy is capable of  treating efficiently high-dimensional control problems, and is robust in the case of limited access to the state and can be implemented with a small number of updates.

	The rest of the paper is organized as follows. In Section \ref{sec:Riccatiderivation} we derive different control systems based on the linearization of the dynamics and the solution of the Riccati equation associated to the linear-quadratic optimal control problem. Section \ref{sec:mean-field} is devoted to the mean-field approximation of the microscopic dynamics and presents bounds  for the moments decay. In Section \ref{sec:MdPC} the Moment-driven Predictive Control framework is described and two different implementations are presented. Finally, in Section \ref{sec:Numerics} we assess the proposed design via numerical experiments, showing different applications in the context of opinion formation and alignment dynamics.
	
	\section{Control of an interacting multi-agent system}\label{sec:Riccatiderivation}
	In this section we present a linearization-based approach for the control of large-scale interacting particle systems. We are concerned with the evolution of $N$ interacting agents, whose states $v_i(t)\in \R^{d}$ evolve according to the following nonlinear model \eqref{nonlin_dynamics}. Further assumptions regarding the interaction kernel $P(u,v)$ governing these interactions will be discussed in the forthcoming sections. The term $u_i\in \R^d$ represents an external control variable acting over the $i-$th agent of the system. The complete set of control variables is denoted by $u=(u_1,\ldots,u_N)\in \R^{N\times d}$. In order to synthesize this control variable, we assume that $u$ is the minimizer of a cost function $J(u;v(0))$, that is
	\begin{align}\label{eq:func}
		u^* = \arg\min_{u} J(u;v^0):= \int_0^T\ell(v(t),u(t))\,dt\,,\qquad\text{subject to \eqref{nonlin_dynamics}}\,.
	\end{align}
    The optimization horizon $T$ expresses the time scale along which we minimize the running cost $\ell(v,u)$, encodes our objective as a function of the state and control variables. In the context of this work, we are interested in {\em consensus equilibrium}, namely, reaching a consensus velocity $\tilde{v}\in\R^d$ such that $v_{1}=\ldots=v_{N}=\tilde{v}$. With a slight abuse of notation, we shall denote indistinctively by $\tilde{v}$ the consensus state and the swarm configuration $\tilde{v}\in\R^{N\times d}$ where all the agents share the same velocity. In order to promote consensus emergence, we solve the optimal control problem \eqref{eq:func} to determine a control law $u$ driving the system towards $\tilde{v}$ using the running cost
	\begin{align}\label{eq:runcost}
		\ell(v,u) = \frac{1}{N}\sum_{j=1}^N ( \vert v_j - \tilde{v} \vert^2 + \nu \vert u_j\vert^2) 
	\end{align}
	where $\nu>0$ is a penalization parameter for the control energy, and $\tilde{v}$ is a prescribed consensus point. The norm $\vert \cdot \vert$ is the usual Euclidean norm in $\mathbb{R}^d$.

	\subsection{Linearization and the LQR approach for collective dynamics}
	
	We are interested in the synthesis of a feedback control law for the control of the non-linear dynamics \eqref{nonlin_dynamics}. We begin by defining the vector-valued function $F(v): \R^{N\times d}\rightarrow \R^{N\times d}$ such that
	\begin{equation}
		F_i(v) = \frac{1}{N}\sum_{j=1}^N P(v_i,v_j)(v_j-v_i),\qquad i=1,\ldots,N.
	\end{equation}
	We linearize the dynamics around $v_i=\bar{v}$ for every agent, which corresponds to an arbitrary equilibrium for the nonlinear dynamics \eqref{nonlin_dynamics}, i.e. $F(\bar v) = 0$, and which can be different from the consensus state $\tilde v$. We further assume that the communication function $P(v_i,v_j)$ is such that $P(\bar v,\bar v ) \equiv \bar p\,,$ with $\bar p$ a bounded value. Computing the first order approximation of $F(v)$ around $\bar v $, we have $\nabla_v F(\bar{v}) (v-\bar{v}) =  A (v-\bar v)\,,$ where $A\in\R^{N\times N}$ is the Laplacian matrix defined as follows
	\begin{align} \label{AB}
		(A)_{ij}=
		\begin{cases} &a_d=\frac{\bar{p}(1-N)}{N},\qquad i=j,\\
			&a_o=\frac{\bar{p}}{N},\qquad\qquad i\neq j.\\
		\end{cases}
	\end{align}
	We observe that the structure of matrix $A$ is such that $A(v-\bar v)=Av$ since $\bar v$ is a consensus point. To write the linearized system associated to \eqref{nonlin_dynamics} we further consider the change of variables $w_i(t):=v_i(t)-\bar v$, and we have
	\begin{equation}\label{lin_dynamics}
		\dot{w}_i = \frac{1}{N}\sum_{j=1}^N \bar{p}(w_j-w_i)+ u_i,\qquad\quad w_i(0)=v_{i}^0-\bar v.
	\end{equation}
	
For the linearized dynamics we cast as the {\em Linear Quadratic Regulator} (LQR) control problem, where the functional \eqref{eq:func} reads as follows
	\begin{align}\label{eq:lin_runcost}
		J(u,w(0)) = \int_0^T w^\top Q w+ \nu u^\top R u \, dt
	\end{align}
	in the matrix-vector notation with $w=(w_1,\ldots,w_N)$ and matrices $Q \equiv R = \frac{1}{N}\textrm{Id}\in\R^{N\times N}$. The linear dynamics \eqref{lin_dynamics} are equivalent to
	\begin{align}\label{eq:lin_dynamics}
		\dot {w} = A w + B u,\qquad w(0) = v^0-\bar v\,,
	\end{align}
	where $B=\textrm{Id}\in\R^{N\times N}$ is the identity matrix, impliying that the pair $(A,B)$ is controllable for any consensus state $\bar v$. Thus, in a neighbourhood of any constant state $\bar v$ the non linear system \eqref{nonlin_dynamics} admits a continuous stabilizing feedback, see e.g. \cite{bressan2007introduction}. In order to synthesize a stabilizing control law we solve the optimal control problem \eqref{eq:lin_runcost}--\eqref{eq:lin_dynamics}, whose exact solution is given in feedback form by
	\begin{align}\label{eq:Riccati_ctrl}
		u(t) = -\frac{N}{\nu} K(t)w(t)
	\end{align}
	with $K(t)\in \R^{N\times N}$ fulfilling the Differential Riccati matrix-equation
	\begin{equation}\label{eq:Riccati}
		- \dot{K} = KA+A^\top K-\frac{N}{\nu} KK + Q, \quad K(T) = 0\in\R^{N\times N},
	\end{equation}
	coupled to the evolution of the controlled system \eqref{eq:lin_dynamics}.	For a general linear system we need to solve the $N\times N$ differential system \eqref{eq:Riccati}, which can be costly for large-scale agent-based dynamics. However, in this case we can exploit the symmetric structure of the Laplacian matrix $A$ to reduce the Riccati equation.
	
	\begin{prop}[Properties of the Differential Riccati Equation] 
		For the linear dynamics \eqref{lin_dynamics}, the solution of the Riccati equation \eqref{eq:Riccati} reduces to the solution of
		\begin{subequations}\label{eq:kd_ko}
			\begin{align}
				-\dot{k_d} &= 2k_da_d + 2(N-1)k_oa_o - \frac{N}{\nu}\left(k_d^2+(N-1)k_o^2\right) + \frac{1}{N}, 
				\\
				-\dot{k_o} &= 2(N-2)k_oa_o+2k_oa_d+2k_da_o - \frac{N}{\nu}\left(2k_dk_o+(N-2)k_o^2\right), 
			\end{align}
		\end{subequations}
	with terminal conditions $k_d(T)=k_o(T)=0$. The solution $K$ of the differential Riccati equation \eqref{eq:Riccati} corresponds to $(K)_{ij}=\delta_{ij}k_d+(1-\delta_{ij})k_o$.
	\end{prop}
	\begin{proof}
	Given the structure of the matrices $K$, $A$ and $Q$,  solving the Riccati equation \eqref{eq:Riccati} componentwise leads to the following identities:
		\begin{align*}
			\text{Diagonal entries $k_{ii}$}:\qquad& \frac{N}{\nu} (K^2)_{ii} = \frac{N}{\nu}\left(k_d^2+(N-1)k_o^2\right).\\
			\text{Off-diagonal entries $k_{ij}$:}\qquad& \frac{N}{\nu} (K^2)_{ij} = \frac{N}{\nu}\left(2k_dk_o+(N-2)k_o^2\right).
		\end{align*}
	\end{proof}
	
	We can further simplify the Riccati-matrix system \eqref{eq:kd_ko} using the dependency of coefficients $a_d,a_o$ \eqref{AB} and the parameter $\bar p$. This leads to
	\begin{align}
		-\dot{{k}}_d &= -\frac{2\bar{p}(N-1)}{N}({k}_d -k_o) - \frac{N}{\nu}\left({k}_d^2+(N-1){k}_o^2\right) + \frac{1}{N}, \qquad  k_d(T)=0 \label{kd0}\,,
		\\
		-\dot{{k}}_o &=  \frac{ 2\bar{p}}{N}(k_d-k_o) - \frac{N}{\nu}\left(2{k}_d {k}_o+(N-2){k}_o^2\right),\qquad  k_o(T)=0. \label{ko0}
	\end{align}

	Since we are interested in the dynamics for large number of agents, we introduce the following scalings
	\begin{equation}
		{k}_d \leftarrow N k_d, \quad {k}_o \leftarrow N^2 k_o,\quad \alpha(N) = \frac{N-1}{N}.
	\end{equation}
	 For the sake of simplicity, we keep the same notation also for the scaled variables $\kd,\ko$. Under this scaling the system \eqref{kd0}--\eqref{ko0} reads 
		\begin{align}
		-\dot\kd &= -2\bar{p}\alpha(N)\left(\kd - \frac{\ko}{N}\right) - \frac{1}{\nu}\left(\kd^2+\frac{\alpha(N)}{N}\ko^2\right) + 1, \qquad k_d(T)=0, \label{kd}
		\\
		-\dot\ko &=   2\bar{p}\left(\kd - \frac{\ko}{N}\right) - \frac{1}{\nu}\left(2\kd \ko+\alpha(N)\ko^2-\frac{1}{N}\ko^2\right),\qquad k_o(T)=0\,, \label{ko}
\end{align}	
	and the Riccati feedback law \eqref{eq:Riccati_ctrl} is given by
	\begin{align}
		u_i &= - \frac{1}{\nu}\left(\left(\kd-\frac{\ko}{N}\right) w_i(t) + \frac{\ko}{ N} \sum_{j=1}^N w_j(t)\right)\,.
	\end{align}
	Plugging the control  into the linear dynamics \eqref{lin_dynamics}  and rearranging the terms we have 
	\begin{equation}\label{RiccatiLinearDynamics}
		\begin{aligned}
			\dot{w}_i& = \left(\bar p -\frac{ k_o}{\nu}\right)\frac{1}{N}\sum_{j=1}^N w_j - \left(\bar p + \frac{\kd}{\nu} - \frac{\ko}{\nu N} \right)w_i,\qquad w_i(0)=v_i^0-\bar v.
		\end{aligned}
	\end{equation}
	
		The controlled dynamics \eqref{RiccatiLinearDynamics} are non-autonomous as the coefficients $ k_d(t), k_o(t)$, have to be determined offline by solving \eqref{kd}--\eqref{ko} backwards in time.
		
		In order to analyse the large-scale behaviour of the system we introduce the average of the agent states, and a weighted combination of the Riccati coefficients, respectively
		\begin{align*}
			m_w^N(t):=\frac1 N\sum_{j=1}^N w_j(t),\quad
			s(t) := \kd(t) +\alpha(N) \ko(t).
		\end{align*}
	 From \eqref{RiccatiLinearDynamics}, \eqref{kd} and \eqref{ko}, these quantities are governed by
		\begin{subequations}\label{eq:micro_average}
			\begin{align}
				\dot m_w^N(t)  &= -\frac{1}{\nu} s(t) m_w^N(t),\qquad m^N_w(0)=m^N_v(0)-\bar v,\\
				-\dot s(t)     &= 1 - \frac{1}{\nu} s(t)^2,\qquad\quad s(T) = 0.
			\end{align}
		\end{subequations}
	 The second equation has an explicit solution $s(t) = \sqrt{\nu}\tanh((T-t)/\sqrt{\nu})$, always non-negative for $t\in[0,T]$ \cite{herty2015mean}. The average $m^N_v(t)$ follows a relaxation towards  $\bar v$
		\begin{align}\label{eq:ave}
			\dot m_v^N(t)  = -\frac{1}{\sqrt{\nu}} \tanh\left(\frac{T-t}{\sqrt{\nu}}\right) (m_v^N(t)-\bar v).
		\end{align}

	\begin{remark}[Second order dynamics]\label{rmk_second order} This approach can be extended to second order models \cite{cucker2007emergent,d2006self,motsch2014heterophilious}. Here, the state space of a swarm of $N$ agents is characterized by position and velocities $(x_i(t),v_i(t))_i\in \mathbb{R}^{2\times d}$, evolving according to 
		\begin{equation}\label{order2}
				\dot{x}_i= v_i, \qquad \dot{v}_i= \frac{1}{N}\sum_{j=1}^N P(x_i,x_j)(v_j-v_i)+ u_i,\qquad i = 1,\ldots,N\,.
		\end{equation}
		where $u\in \mathbb{R}^{N\times d}$. We consider again a functional of type \eqref{eq:runcost}, where we enforce a consensus point $v_i=v_j=\tilde{v}$ for every $i,j$. Linearizing around this point and introducing the shift $y_i = x_i-\bar v t, \quad w_i = v_i -\bar v$, the system is transformed into
		\begin{equation}
			\begin{bmatrix} \dot y \\ \dot w \end{bmatrix}
			= \begin{bmatrix} 0 & \textrm{Id} \\ 0 & A \end{bmatrix} \begin{bmatrix} y \\ w \end{bmatrix}
			+ \begin{bmatrix} 0 \\\textrm{Id} \end{bmatrix} u\,.
		\end{equation}
		This second-order system system is controllable \cite{herty2018suboptimal}, and the associated Differential Riccati Equation reads
		\begin{equation}
			\begin{aligned}
				\begin{bmatrix} \dot K_{11} & \dot K_{12} \\ \dot K_{21} & \dot K_{22} \end{bmatrix}
				= &\begin{bmatrix} 0 & 0 \\ 0 & K_{21} + K_{22}A \end{bmatrix}
				+ \begin{bmatrix} 0 & 0 \\ K_{11} + A K_{21} & K_{12} + AK_{22} \end{bmatrix} \\
				&- \frac{N}{\nu} \begin{bmatrix} K_{12} K_{21} & K_{12} K_{22} \\ K_{22} K_{21} & (K_{22})^2 \end{bmatrix}
				+ \begin{bmatrix} 0 & 0 \\ 0 & \textrm{Id} \end{bmatrix} ,
			\end{aligned}
		\end{equation}
		with terminal conditions $K_{ij}(T) = 0$, for $\ i,j = 1,2$. This system is easily solved with $K_{11} = K_{12} = K_{21} = 0$ and $K_{22}$ satisfying a Riccati equation equivalent to \eqref{eq:Riccati}. Hence, the results we obtain for the first order system can be extended to second order systems. We will further discuss this extension in the numerical section.
	\end{remark}

	\subsection{Riccati-based control laws for the non-linear system} \label{Riccati controls for the non-linear system}
	In order to approximate the synthesis of feedback laws for the original nonlinear optimal control problem \eqref{eq:func}, we study sub-optimal stabilizing strategies induced by the Riccati control \eqref{eq:Riccati_ctrl}. 
	Without loss of generality, we will consider the stabilization problem towards $\tilde{v} = 0$. We will focus on different ways to synthesize a control law based on the information retrieved in the linearized case: a closed-loop control, an open-loop strategy, and a simplified inexact open-loop control. 
	
	\paragraph{Closed-loop control} A well-known local strategy for the control of nonlinear dynamics is to use the optimal feedback control obtained from the linearized dynamics. In this case, the controlled system reads
	\begin{equation}\label{FBA}
		\dot{v}_i = \frac{1}{N}\sum_{j=1}^N P(v_i,v_j)(v_j-v_i) + R_i[v](t),\qquad v_i(0)=v^0_i,
	\end{equation}
	where the operator $R_i[\cdot]$ is the feedback \eqref{eq:Riccati_ctrl} applied directly to the state of the non-linear system $v(t)=(v_i(t))_{i=1}^N$, namely
	\begin{align}
		R_i[v](t) &= - \frac{1}{\nu}\left(\left(\kd(t)-\frac{\ko(t)}{N}\right) v_i(t) + \frac{\ko(t)}{N} \sum_{j=1}^N v_j(t)\right),
	\end{align}
	where $\kd(t),\ko(t)$ are still obtained by solving the system \eqref{kd}--\eqref{ko}. In general, such a control law is expected to work only for initial states sufficiently close to the state around which the dynamics have been linearised. We shall investigate in detail the properties of the closed-loop in the following section.
	\paragraph{Open-loop control} The open-loop strategy we propose applies the control signal obtained from the linear synthesis $u_i(t;v^0)$ directly into the  the non-linear dynamics as follows
	\begin{subequations}\label{FFA} 
		\begin{align}
			\dot{v}_i &= \frac{1}{N}\sum\limits_{j=1}^N P(v_i,v_j)(v_j-v_i) +  R_i[w](t), \qquad v_i(0)=v^0_i,
			\\
			\dot{w}_i &= \frac{1}{N}\sum\limits_{j=1}^N\bar p(w_j-w_i) + R_i[w](t),\qquad\quad \qquad w_i(0)=v^0_i,
		\end{align}
	\end{subequations}
	where the control $ R_i[w](t)$ is computed according to \eqref{eq:Riccati_ctrl}.
	This approach is open-loop, since all the information on the state of the non-linear system reduces to the initial state of linearized system, assuming $w^0_i=v^0_i$. While this approach is clearly outperformed by the closed-loop feedback law in terms of robustness, it has the advantage that it can be implemented without requiring a continuous measurement of the full nonlinear state $v(t)$, making it appealing for systems where recovering the true state of the dynamics can be expensive or time-consuming.
	\paragraph{Inexact open-loop control}
	An inexact, but simpler, implementation of the open-loop approach \eqref{FFA} obtained when the control $R_i[\cdot]$ is evaluated only with respect to the initial data $v^0_i$, that is $R_i[v^0](t)$, where the dependence on $t$ is limited to $\ko(t)$ and $\kd(t)$. This setting avoids the evaluation of system in \eqref{FFA}, and only requires the computation of 
	\begin{equation}\label{FFA_0} 
		\dot{v}_i = \frac{1}{N}\sum_{j=1}^N P(v_i,v_j)(v_j-v_i) +  R_i[v^0](t), \qquad v_i(0)=v^0_i.
	\end{equation}
	The open-loop control laws are meant to be embedded in a Model Predictive Control framework, ensuring a sufficiently frequent update of the state of the nonlinear system ensuring stability of the resulting control system. This shall be further analysed in Section \ref{sec:MdPC}. In the following section we will study the performance of these control strategies when stabilizing the non-linear dynamics in the case $N\gg1$.

	\section{Mean-field limits and moments estimates}\label{sec:mean-field}
	
	The stabilization strategies \eqref{FBA} and \eqref{FFA} are clearly suboptimal with respect to the original optimal control problem, and in general will not guarantee the stabilization of the non-linear dynamics \eqref{nonlin_dynamics}. In this section we quantify the discrepancy between the desired target state and the final state obtained by the stabilization strategies  \eqref{FBA} and \eqref{FFA}. In order to estimate these performances in the case where a large number of agents is present, i.e. $N\gg 1$, we discuss our approaches in the {\em mean-field limit}.  Here, we consider the density distribution of agents in order to describe the collective behavior of the ensemble of particles, and we retrieve upper and lower bounds for the decay of the mean-field density towards the desired configuration.

	\subsection{Open-loop Riccati control}
	We introduce the empirical joint probability distribution of particles for the system \eqref{nonlin_dynamics} and \eqref{lin_dynamics} is given by
	\begin{equation}
		\lambda^N(t,v,w) = \frac{1}{N} \sum_{i=1}^N \delta (v-v_i(t)) \delta(w-w_i(t)),
	\end{equation}
	where $\delta(\cdot)$ is a Dirac measure. We assume enough regularity on the interaction kernel, assuming that particles remain in a fixed compact domain for all $N$ and in the whole time interval $[0,T]$. We refer to \cite{canizo2011well,carrillo2014derivation} for a rigorous treatment of the mean-field limit of interacting particle systems.
	Hence, we introduce the test function $\phi(v,w) \in C^1_0(\mathbb{R}^{2d})$ and by Liouville's theorem we compute the time variation of the inner-product $\langle \lambda^N(t),\phi\rangle$, given by
	\begin{equation}
		\frac{d}{dt} \langle \lambda^N(t),\phi\rangle = \frac{1}{N} \sum_{i=1}^N (\nabla_v \phi(v_i,w_i) \cdot \dot{v}_i(t) + \nabla_w \phi(v_i,w_i) \cdot\dot{w}_i(t)).
	\end{equation}
Denoting by $\lambda^N(t):=\lambda^N(t,v,w)$, we define the marginal densities $f^N$, $g^N$, the average $m_1[f^N]$ and the second moment $m_2[f^N]$ as follows
	\begin{equation}\label{eq:marginals}
		\begin{aligned}
			f^N(t,w)&:=\int_{\R^{d}} \lambda^N (t,v,w) dv,
	        \qquad g^N(t,v):=\int_{\R^{d}} \lambda^N (t,v,w) dw;
	 \cr
			m_1[f^N](t)&:=\int_{\R^{d}} w f^N(t,w)dw\,,\quad m_2[f^N](t):=\int_{\R^{d}} |w|^2 f^N(t,w)dw\,.
		\end{aligned}
	\end{equation}
	With the standard derivation of the mean-field limit, we obtain in the strong form the evolution equation for $\lambda^N(t,v,w)$ as follows
	\begin{align*}
		\partial_t \lambda^N &= -\nabla_v \cdot \left[\lambda^N \left( \mathcal P[g^N]- a^N(t) w -b^N(t)m_1[f^N]\right)\right]\\
		\quad&\qquad-\nabla_w \cdot \left[ \lambda^N \left(\bar p(m_1[f] - w)- a^N(t) w -b^N(t)m_1[f^N] \right) \right],
	\end{align*}
	where $\mathcal{P}[g]$ denotes the nonlocal integral operator
	\[
	\mathcal{P}[g](v,t)  =\int_{\R^d}P(v,v_*)(v_*-v) g(v_*,t) dv_*,
	\]
	and  $a^N(t)$ and $b^N(t)$ are obtained from the scaled Riccati system \eqref{kd}-\eqref{ko}.
Since we are interested in the limit of a large number of agents, for $N\to\infty$ we have
	\[
	\lim_{N\rightarrow \infty } a^N(t)= \frac{k_d(t)}{\nu},\qquad \lim_{N\rightarrow \infty }b^N(t)= \frac{k_o(t)}{\nu},
	\]
	where $\kd$ and $\ko$ fulfill
	\begin{equation}
		\begin{aligned}\label{ko_kd_limit}
			-\dot\kd=  -2\bar{p}\kd -\frac{k_d^2}{\nu}+1, \qquad
			-\dot\ko=   2\bar{p}\kd- \frac{\ko}{\nu}\left(2\kd+\ko\right)\,.
		\end{aligned}
	\end{equation}
Then, the joint mean-field model can be conveniently written as
	\begin{equation}
		\begin{aligned}\label{eq:jointmean-field}
			\partial_t \lambda &= -\nabla_v \cdot\left[ \lambda \left( \mathcal P[g]- \frac{k_d}{\nu} w -\frac{k_o}{\nu}m_1[f]\right)\right]\cr
			&\qquad\qquad-\nabla_w \cdot \left[\lambda\left(\bar p\left(m_1[f]-w\right) - \frac{k_d}{\nu} w -\frac{k_o}{\nu}m_1[f])\right) \right],
		\end{aligned}
	\end{equation}
	with intial data $\lambda(v,w,0) = \lambda^0(v,w)$.
	Integrating  the mean-field  equation \eqref{eq:jointmean-field} with respect to $v$ and $w$, the evolution of the  marginals $f(w,t)$ and $g(v,t)$ is given by
	\begin{subequations}\label{eq:mean-field_OL}
		\begin{align}\label{eq:mean-field_nonlinear}
			&\partial_t g = - \nabla_v \cdot\left[ g\left(\mathcal{P}[g] - \frac{\kd}{\nu} m_1[h] -\frac{\ko}{\nu}  m_1[f]\right) \right], \qquad\qquad\,\, g(v,0)=g^0(v) \\
			&\partial_t f = - \nabla_w \cdot \left[f \left(\left(\bar p-\frac{k_o}{\nu}\right)m_1[f] - \left(\bar p+\frac{k_d}{\nu}\right)w\right)\right],\qquad f(w,0)=g^0(w)\label{eq:mean-field_Linear}
		\end{align}
	\end{subequations}
 with $m_1[h](v,t)$ being the average of the conditional probabilty $h(w|v,t)$ defined as
	\begin{align}\label{eq:cond}
		\lambda(t,v,w) = h(t,w|v)g(t,v),\qquad m_1[h](t,v) = \int_{\R^d} w h(t,w\vert v) dw.
	\end{align}
	We observe that equation \eqref{eq:mean-field_Linear} is the mean-field equation associated to the linear controlled model \eqref{RiccatiLinearDynamics}  \cite{herty2015mean}. On the other hand, the equation for the non-linear model \eqref{eq:mean-field_nonlinear} is coupled to the solution of the linear model through the control.  
\begin{remark}[Inexact mean-field open-loop control]
		If we consider the inexact open-loop control  \eqref{FFA_0}, where the control acts only on the information measured at the initial time $t_0$,
		we derive a consistent mean-field limit following the procedure described in this section. The mean-field equation for $\lambda(v,w,t)$ reads 
		\begin{equation}
			\begin{aligned}\label{eq:jointmean-field_0}
				\partial_t \lambda &= -\nabla_v \cdot\left[ \lambda \left( \mathcal P[g]- \frac{k_d}{\nu} w -\frac{k_o}{\nu}m_1[g^0]\right)\right],\qquad \lambda(v,w,0) = \lambda^0(v,w).
			\end{aligned}
		\end{equation}
	The marginal distributionscorresponds to the initial data $g^0(w)$ and $g(t,v)$ respectively, and the system \eqref{eq:mean-field_OL} reduces then to the equation
		\begin{align}\label{eq:mean-field_nonlinear_0}
			&\partial_t g = - \nabla_v \cdot\left[ g\left(\mathcal{P}[g] - \frac{\kd}{\nu} m_1[h] -\frac{\ko}{\nu}  m_1[g^0]\right) \right], \qquad\,\, g(v,0)=g^0(v).
		\end{align}	
	\end{remark}
	\subsubsection{Bounds of decay moments}
	We are interested in the evolution of the first moment and variance of the nonlinear mean-field density $g(t,v)$, denoted by $m_1[g](t)$ and $\sigma^2[g](t)$, respectively. Stabilizing the system towards target consensus point $\bar v =0$ requires estimates on the decay of moments towards zero.  We  assume the kernel $P(\cdot,\cdot)$ to be a symmetric and bounded function, namely
	\begin{align}\label{eq:assumption1}\tag{\textbf{A}}
		P(v,v_*) = P(v_*,v),\qquad P(v,v_*)\in [-a,b],\qquad \forall v,v_*\in\R^{d}, \qquad a,b\geq 0.\qquad 
	\end{align}
    Hence we have the following results
    	\begin{lem}\label{lem:sigmaFFA} Let assumption \eqref{eq:assumption1} holds for the interaction kernel $P(\cdot)$, then the average of non-linear model \eqref{eq:mean-field_nonlinear} decays as follows
    		\begin{align}\label{eq:mf_linear_average}
    		\frac{d}{dt}m_1[g] = -\frac{\kd+\ko}{\nu}m_1[g],\qquad m_1[g](0) = m_1[g^0],
    		\end{align}
    		and the average of the linear model $m_1[f](t)$ coincides with $m_1[g](t)$ for $t\geq0$.
    		
    		The variance variance $\sigma^2[g](t)$ evolves according to
    			\begin{equation} \label{sigma}
    		\frac{d}{dt}\sigma^2[g] = -\int_{\R^{2d}}\!\!\!\!\!\vert v-v_*\vert^2 P(v,v_*) g(v) g(v_*) dv dv_* - 2 \frac{\kd}{\nu} \varrho[f,g] \sqrt{\sigma^2[f] \sigma^2[g]},
    		\end{equation}
    		with intial data $\sigma^2[g](0)=\sigma^2[g^0]$, where $\varrho[f,g]$ is the correlation coefficient. Moreover the variance of the linear model $\sigma[f]$ satisfies
    	    \begin{equation}\label{sigmaf}
    		\frac{d}{dt}{\sigma}^2[f]  = -2 \Big( 2 \bar{p} + \frac{\kd}{\nu} \Big)\sigma^2[f],\qquad \sigma[f^0]=\sigma[g^0].
    		\end{equation}
    		\end{lem}
	\begin{proof}
       By construction of model \eqref{eq:mean-field_OL} the first moment at time zero coincides, i.e. $m_1[f^0] = m_1[g^0]$. The first moment of $g(t,v)$ satisfies
	\begin{align*}
		\frac{d}{dt}m_1[g] &= \int_{\R^{2d}}\!\!\!\!\!P(v,v_*)(v_*-v)gg_*dvdv_* -\frac{\kd}{\nu}\int_{\R^d} m_1[h](v)g\ dv-\frac{\ko}{\nu}m_1[f]\cr
		& = -\frac{\kd}{\nu}\int_{\R^{2d}}\!\!\!\!\!w \lambda(v,w)\ dvdw-\frac{\ko}{\nu}m_1[f]=-\frac{\kd+\ko}{\nu}m_1[f],
	\end{align*}
	where we omitted time dependencies and we denote by $gg_*$ the product of $g(t,v)$ and $g(t,v_*)$.	The last two equalities follow from the symmetry of $P(\cdot)$ and the definition of joint distribution. 
	
	The second moment of $g(t,v)$, $m_2[g](t)$ satisfies the following equation
	\begin{align*}
		\frac{d}{dt}{m}_2[g]
		=&\! -\int_{\R^{2d}}\!\!\!\!\!\vert v-v_*\vert^2 P(v,v_*) g g_* dv dv_* - 2 \frac{\ko}{\nu}m_1[g]m_1[f] -2\frac{\kd}{\nu} \int_{\R^{2d}}\!\!\!\!\!v w \lambda(v,w) dv dw,
		\\
		=& -\!\int_{\R^{2d}}\!\!\!\!\!\vert v-v_*\vert^2 P(v,v_*) g g_* dv dv_*- 2\frac{\kd\!+\!\ko}{\nu} |m_1[g]|^2 \!\!- 2\frac{\kd}{\nu}\varrho[f,g] \sqrt{\sigma^2[f] \sigma^2[g]}
	\end{align*}
	where we used the equivalence of $m_1[g]\equiv m_1[f]$ and the relation between the correlation coeffiecient $\varrho[f,g]$ and the covariance between $f$ and $g$, namely
	\begin{align*}
	\varrho[f,g] \sqrt{\sigma^2[f] \sigma^2[g]} =	\int_{\R^{2d}}\!\!\!\!\!v w \lambda(v,w) dv dw - m_1[f] m_1[g].
	\end{align*}
By observing that 
\begin{equation*}
		\frac{d}{dt}|m_1[g]|^2 = - 2 \frac{\kd + \ko}{\nu} |m_1[g]|^2
\end{equation*}
and using the definition $\sigma^2[g] = m_2[g] - \vert m_1[g] \vert^2$, we obtain the equation for the variance of the non-linear model \eqref{sigma}.
The variance for the linear model is obtained directly from \eqref{sigma}, imposing $P(v,v_*)=\bar p$ and observing that the double integral becomes
\begin{equation}\label{equality2}
\int_{\R^{2d}}\!\!\!\!\!\vert w-w_*\vert^2 f f_* dw dw_* = 2 \sigma^2[f].
\end{equation}
\end{proof}
From the evolution equation of the variance $\sigma[g](t)$ we retrieve  the following estimates on the decay.
	\begin{prop}\label{prop:Bounds}
		Under assumption \eqref{eq:assumption1} on the interaction kernel $P(\cdot)$, we have the following lower and upper bounds for the evolution of the variance $\sigma^2[g]$:
		\begin{equation}\label{ul_bounds}
			\begin{split}
				\sigma^2[g^0]e^{-2bt} \left(1 - B^+_b(0,t)\right)^2
				\leq
				\sigma^2[g](t)
				\leq
				\sigma^2[g^0]e^{2at} \left( 1 + B^-_a(0,t) \right)^2,\quad\text{where}
			\end{split}
		\end{equation}
		\begin{equation} \label{B}
			\begin{split}
				B^\pm_c(t_0,t) = \frac{1}{\nu}\int_{t_0}^{t-t_0}  \beta(s-t_0)\kd(s)e^{\pm c (s-t_0)} ds,\cr
				\beta(t-t_0) = \exp\left\{-2 \bar{p}(t-t_0)-\frac{1}{\nu} \int_{t_0}^{t-t_0} \kd(r) dr\right\}.
			\end{split}
		\end{equation}
	\end{prop}

	\begin{proof}			
		 Consider first the case $P(v,w) \geq -a$.
			We bound from below the interaction kernel in equation \eqref{sigma}, 
			\begin{align*}
				\frac{d}{dt}{\sigma}^2[g](t)& \leq 2 a\int_{\R^{2d}}\!\!\!\!\!|v-v_*|^2gg_*dvdv_* - 2 \frac{\kd}{\nu} \varrho[f,g] \sqrt{\sigma^2[f] \sigma^2[g]}\cr
				& \leq 2 a \sigma^2[g] - 2 \frac{\kd}{\nu} \sqrt{\sigma^2[f] \sigma^2[g]}.
			\end{align*}
			where we first used the identity \eqref{equality2} and $|\varrho|\leq 1$.
			In order to estimate the growth of the right hand side we note that because of \eqref{sigmaf}, $\sigma[f](t)$ is given by
			\begin{align*} 
				\sigma^2[f](t) =& \sigma^2[g^0] \exp{\left\{-4\bar{p}t-\frac{2}{\nu}\int_0^t \kd(s)ds\right\}}=: \sigma^2[g^0]\beta(t)^2.
			\end{align*}
			 Substituting the estimate in the previous equation we obtain
			\begin{align*}
				\frac{d}{dt}{\sigma}^2[g](t)
				\leq& 2 a\sigma^2[g](t) + \frac{2\kd}{\nu}\beta(t)\sqrt{\sigma^2[g^0]\sigma^2[g](t)}.
			\end{align*}
			An estimate can be obtained for $z \neq 0$:
			\begin{align*}
				\frac{d}{dt}z(t)
				=& az(t) + \frac{\kd}{\nu}\beta(t)\sqrt{\sigma^2[g^0]}.
			\end{align*}
			This first-order linear differential equation admit an exact solution as follows
			\begin{equation} \label{z_eq}
				\begin{split}
					z(t) &= z(0) e^{at} \left( 1 + \frac{1}{\nu} \int_{0}^t e^{-as} \kd(s) \beta(s) ds
					\right).
				\end{split}
			\end{equation}
			Applying the Petrovitsch's theorem \cite{Petrovitch1901}, we obtain the upper variance bound
			\begin{equation}
				\begin{split}
					\sigma^2[g](t) &\leq \sigma^2[g^0]e^{2at} \left( 1 + \frac{1}{\nu} \int_{0}^t e^{-as} \kd(s) \beta(s) ds
					\right)^2.
				\end{split}
			\end{equation}

			We continue with the case $P(v,w) \leq b$. Bounding from above $P(\cdot,\cdot)$ in \eqref{sigma} we have
			\begin{align*}
				\frac{d}{dt}{\sigma}^2[g](t)
				& \geq -2 b\sigma^2[g](t) - 2\frac{\kd}{\nu}\sqrt{\sigma^2[f]
					\sigma^2[g]}.
			\end{align*}
			Solving exactly \eqref{sigmaf}, and substituting  $\sigma[f](t)$ into the previous equation we obtain
			\begin{align*}
				\frac{d}{dt}{\sigma}^2[g](t)
				\geq& -2 b\sigma^2[g](t) - \frac{2\kd}{\nu}\beta(t)\sqrt{\sigma^2[g^0]\sigma^2[g](t)}.
			\end{align*}
			Proceeding as in equation \eqref{z_eq} leads to
			\begin{equation}
				\begin{split}
					\sigma^2[g](t) &\geq \sigma^2[g^0]e^{-2bt} \left( 1 - \frac{1}{\nu} \int_{0}^t e^{bs} \kd(s) \beta(s) ds
					\right)^2.
				\end{split}
			\end{equation}
		
	\end{proof}
	
	Figure \ref{fig:FFAbound} shows two examples for the decay of $\sigma^2[g]$ and the bounds for the kernel
	\begin{equation}\label{eq:kernel_P}
		P(v,w) = \alpha + \frac{K}{(\varsigma+|v-w|^2)^{\gamma}},\qquad \alpha,\varsigma, \gamma\geq0,\textrm{ and } K\in\mathbb{R}\,,
	\end{equation}
associated to Cucker-Smale consensus dynamics \cite{cucker2007emergent}. On the left, we consider the attractive case where $P(\cdot)$ is positive and bounded in $[0,1]$,  with $\alpha = 0, \varsigma = 1, K=1$ and $\gamma=2$. On the right, we show an attraction-repulsion dynamics with kernel $-1\leq P(\cdot,\cdot)\leq 9$  where $\alpha = 9, \varsigma = 0.1, K=-1$ and $\gamma=1$.
	The value of $\sigma_g^2$ is computed integrating numerically a mean-field approximation of the nonlinear dynamic equation. In both cases, the initial density of particle $g^0(v)$ is
	\[g^0(v)=\frac{2}{3}\chi_{[1/4,7/4]}(v).\].	
	\begin{figure}[!ht]
		\begin{center}
			\begin{tabular}{c c}
				\textbf{$0\leq P(v,w)\leq1$} & \textbf{$-1\leq P(v,w)\leq 9$} \\
				\includegraphics[width=0.46\linewidth]{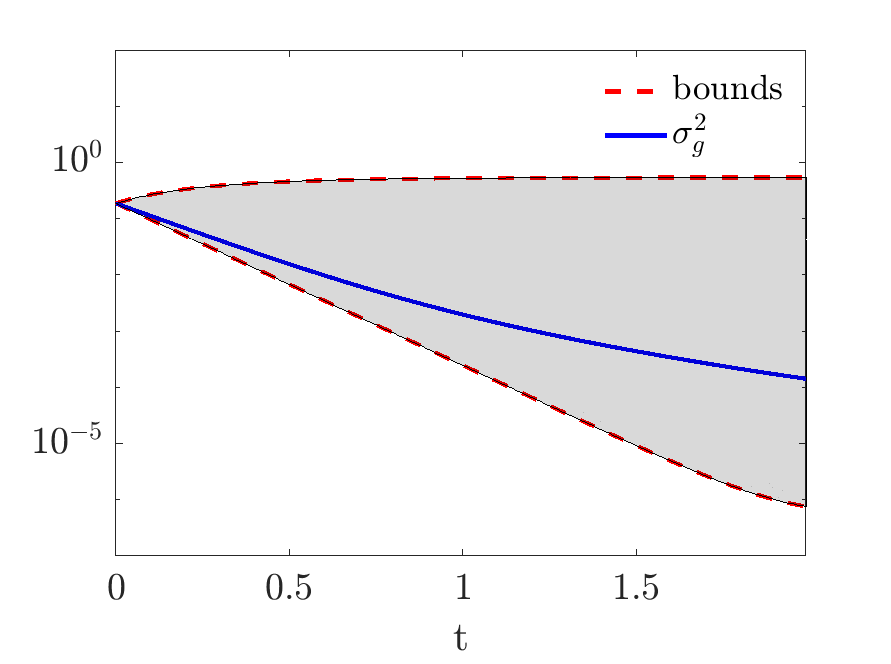} &
				\includegraphics[width=0.46\linewidth]{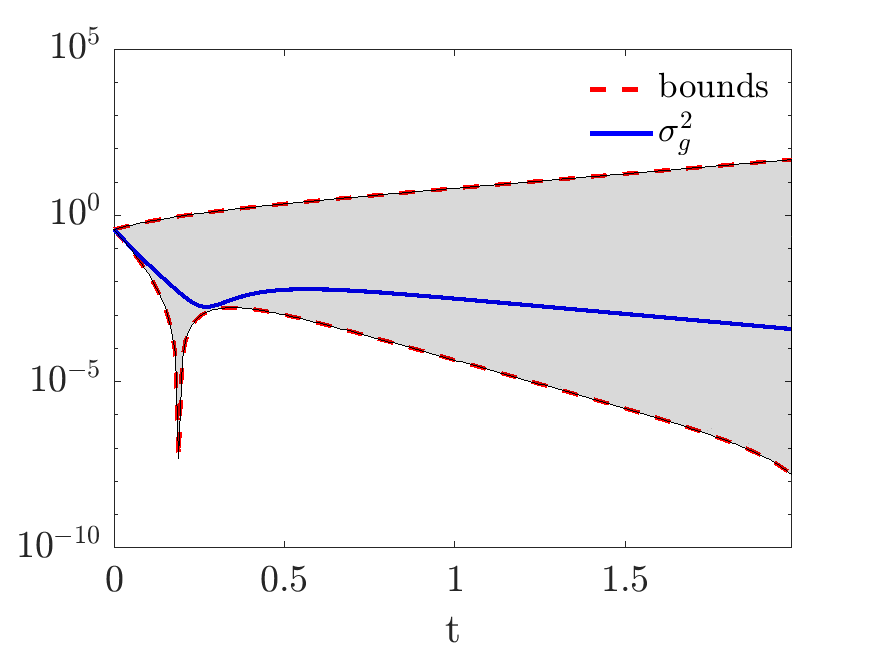}
			\end{tabular}
			\caption{\small{Variance decay and bounds for the open-loop approach \eqref{FFA}. On the left we observe the decay for an attractive dynamics, on the right the attractive-repulsive case.}}\label{fig:FFAbound}
		\end{center}
	\end{figure}
	
	\begin{remark}[Bounds for inexact open-loop Riccati control] 
		For the inexact open-loop Riccati control we observe that bounds on moments can be computed in similar fashion. From \eqref{eq:mean-field_nonlinear_0}, the first moment is given by
		\begin{equation}\label{diff_m1}
			\frac{d}{dt}m_1[g](t) = -\frac{k_d(t)+k_o(t)}{\nu}m_1[g^0],  \qquad m_1[g](0) = m_1[g^0].
		\end{equation}
		A substantial difference with respect to the exponential decay of the average \eqref{eq:mf_linear_average} is observed since the exact solution to \eqref{diff_m1} is
		\begin{equation}\label{eq:ave_0}
			m_1[g](t) = m_1[g^0]\left(1-\frac{1}{\nu}\int_0^t(k_d(s)+k_o(s))\ ds\right).
		\end{equation}
		Bounds of the variance of $g(v,t)$ are retrieved as a particular case of Proposition \ref{prop:Bounds}. The time evolution of $\sigma^2[g]$ reads
		\begin{equation} \label{sigma_0}
			\frac{d}{dt}\sigma^2[g] = -\int_{\R^{2d}}\!\!\!\!\!\vert v-v_*\vert^2 P(v,v_*) g g_* dv dv_* - 2 \frac{\kd}{\nu} \varrho[g^0,g] \sqrt{\sigma^2[g^0] \sigma^2[g]},
		\end{equation}
		and the bounds for $\sigma^2[g](t)$ correspond to the estimate in \eqref{ul_bounds} with $\beta(t)\equiv 1$,
		\begin{equation}\label{ul_bounds_0}
				\sigma^2[g^0]e^{-2bt} \left(1 - \frac{1}{\nu}\int_{0}^{t} e^{bs}  \kd(s)\!\ ds\right)^2\!\!\!
				\leq
				\sigma^2[g](t)
				\leq
				\sigma^2[g^0]e^{2at} \left( 1 + \frac{1}{\nu}\int_{0}^{t}  e^{-as}\kd(s)\!\ ds \right)^2
		\end{equation}
		
		The loss of the exponential decay of the average \eqref{eq:mf_linear_average} constitutes the main drawback of this approach, although we can still steer the density towards a reference solution.
	\end{remark}
	
	\subsection{Closed-loop Riccati control}
	We perform the derivation of the mean-field limit and moment bounds for the system \eqref{FBA}. Given the mean-field density $g(v,t)$, the mean-field limit of \eqref{FBA} is 
	\begin{equation}\label{eq:MF_FBA}
		\partial_t g = - \nabla_v \cdot\left(g\left(\mathcal{P}[g] - \frac{\kd}{\nu} v -\frac{\ko}{\nu}  m_1[g]\right) \right), \qquad\qquad\,\, g(v,0)=g^0(v)\,,
	\end{equation}
where $k_o(t),k_d(t)$ are obtained by a Riccati system \eqref{ko_kd_limit}. 
	\subsubsection{Bounds of decay moments}
	The first moment and variance equations are given in the following Lemma.
	\begin{lem}
	Under the assumption \eqref{eq:assumption1} the first moment of \eqref{eq:MF_FBA} evolves according \eqref{eq:mf_linear_average}.
	The evolution of the variance $\sigma^2[g]$ satisfies the equation
	\begin{equation}\label{sigma_CL}
	\frac{d}{dt}{\sigma}^2[g] = -\iint \vert v-v_*\vert^2 P(v,v_*) g(v) g(v_*) dv dv_* - \frac{2 \kd}{\nu} \sigma^2[g].
	\end{equation}
\end{lem}
    We omit the computations of the proof, since they follow the same line of Lemma \ref{lem:sigmaFFA}. In particular, it is enough to observe that the variance equation of the open-loop Riccati \eqref{sigma} collapses to \eqref{sigma_CL} taking $f(t,v)=g(t,v)$. 
The following estimates on the decay of the variance hold.

	
	\begin{prop} \label{lemma_sigmag*} 
		Under assumption \eqref{eq:assumption1} on the interaction kernel $P(\cdot)$, there exist lower and upper bounds for the variance of $g$ given by
		\begin{equation}
			\sigma^2[g^0] e^{-2bt}C_\nu(0,t) \leq \sigma^2[g](t) \leq \sigma^2[g^0]e^{2at}C_\nu(0,t).
		\end{equation}
		where 
		$$
		C_\nu(0,t) = \exp\left\{-\frac{2}{\nu}\int_0^t \kd(s)ds\right\}.
		$$
	\end{prop}
	\begin{proof}Since the interaction kernel is bounded and using the identity \eqref{equality2} for $\sigma^2[g]$,
		it follows that
		\begin{equation}
			-2 \Big( b + \frac{\kd}{\nu} \Big)\sigma^2[g] \leq \frac{d}{dt}{\sigma}^2[g] \leq 2 \Big( a - \frac{\kd}{\nu} \Big)\sigma^2[g].
		\end{equation}
	\end{proof}

	In Figure \ref{fig:FBAbound} the decay of the variance $\sigma^2[g]$ and the bounds associated to kernel \eqref{eq:kernel_P} are shown. We choose the same parameters as reported in Figure \ref{fig:FFAbound}.
	We  observe that bounds of the closed-loop control \eqref{FBA} are closer compared to equation \eqref{FFA}. Moreover, a stronger decay is observed. Hence, we expect a better performance of the closed-loop control over the open-loop approaches. However, the open-loop approach \eqref{FFA} is useful when dealing with incomplete information or limited access to the non-linear dynamics.

	We devote the next section to the development of a synthesis method based on predictive horizons estimated a-priori through the bounds of the open-loop strategy \eqref{FFA}. 
	
	\begin{figure}[!ht]
		\begin{center}
			\begin{tabular}{c c}
				\textbf{$0\leq P(v,w)\leq1$} & \textbf{$-1\leq P(v,w)\leq 9$} \\
				\includegraphics[width=0.46\linewidth]{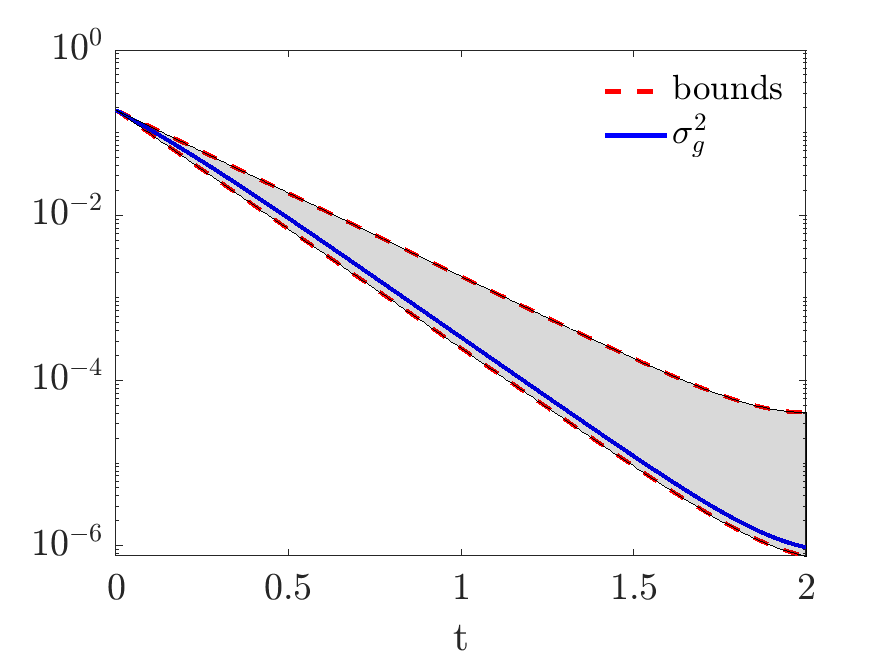} &
				\includegraphics[width=0.46\linewidth]{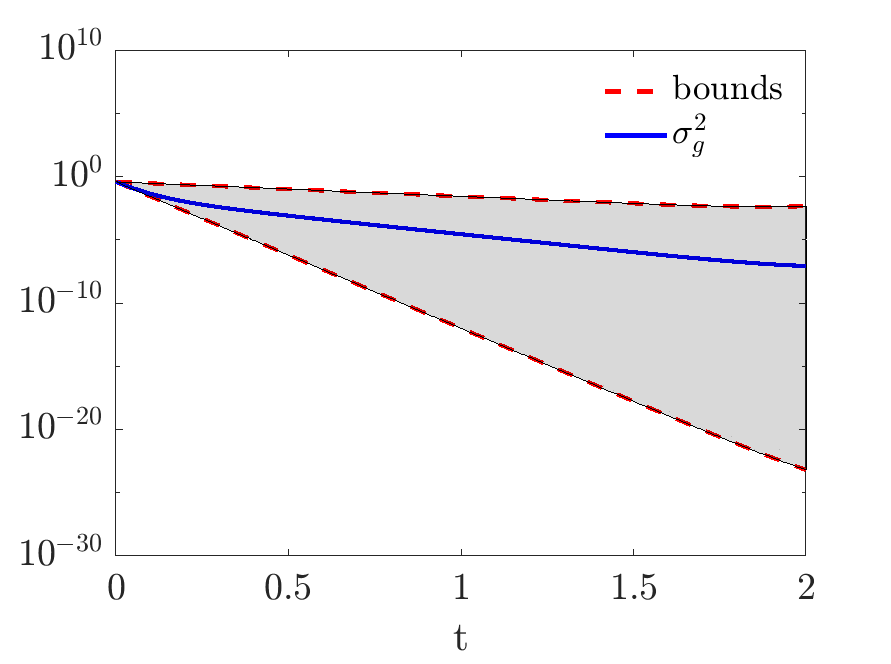}
			\end{tabular}
			\caption{\small{Variance decay and bounds for the open-loop approach \eqref{FBA}. On the left we observe the decay for an attractive dynamics, on the right the attractive-repulsive case.}}\label{fig:FBAbound}
		\end{center}
	\end{figure}

	\section{Moment-driven predictive control (MdPC)} \label{sec:MdPC}
	In order to utilize the stabilization properties of the control loops proposed in the previous sections, we discuss their implementation in a receding horizon framework. Here, we prescribe a control horizon where the control signal is applied, after which there is an update procedure including a re-calculation of the control law based on the current state of the system. There exists a vast literature addressing the design of nonlinear model predictive control (MPC) algorithms, we refer the reader to \cite{herty2017performance,grune2017nonlinear} and references therein. 
	
	In the general nonlinear MPC control algorithm, an open-loop optimal control signal is synthesized over a prediction horizon $[0,T_p]$, by solving a problem of the form \eqref{eq:func}. Having prescribed the system dynamics and the running cost, this optimization problem depends on the initial state $v(0)$ and the horizon $T_p$ only. The optimal signal $u^*$, which is obtained for the whole horizon $[0,T_p]$, is implemented over a shorter control horizon $[0,T_c]$. At $t=T_c$ the initial state of the system is re-calibrated to $v(0)=v(T_c)$ and the optimization is repeated. Relevant issues in the MPC literature are the selection of suitable horizons $T_p$ and $T_c$ which can ensure asymptotic stability of the closed-loop, as well as the design of effective optimization methods to make this implementation suitable for real-time control.
	
	Here instead, we propose a novel class of MPC-type algorithms where instead of fixing a prediction horizon, the re-calibration of the control laws \eqref{FFA}-\eqref{FFA_0} is triggered adaptively in time based on a direct estimate of the moments decay. This is similar in spirit to the literature on event-based MPC methods, see  for example \cite{6315560} where an event-based framework for the control of a team of cooperating distributed agents and \cite{5400921} for networked  systems is proposed.

	\paragraph{Variance driven Predictive Control MdPC($\sigma^2$)} Starting from the open-loop control \eqref{eq:mean-field_OL} we consider densities at initial time given by $g(v,0)=g^0(v)$, $f^0(w)\equiv g^0(w)$  and the joint distribution $\lambda^0(v,w)\equiv \lambda(v,w,t_0)$. To shorten the notation we introduce the general semi-discretization of the mean-field dynamics \eqref{eq:jointmean-field} as follows
	\begin{equation}\label{discreteMF_FFA}
			\lambda^{n+1}(v,w)= \Phi_{\Delta t}[\lambda^n;u^n[f^n]](v,w), \qquad \lambda^0(v,w)=\lambda(v,w,0),\quad n\geq 0,
	\end{equation}
	coupled with the solution of the Riccati system \eqref{ko_kd_limit}. Here, $\Phi_{\Delta t}$ defines the time discretization, and $u^n[f^n]$ encodes the control dependency on the density, given by
	\begin{equation}\label{eq:OCctrl_disc}
		u^n[f^n](w,t_n) = \frac{1}{\nu} \left(\kd(t_n) w +\ko(t_n) m_1[f^n](t_n)\right).
	\end{equation}
	Our goal is to predict the error in the variance decay $\sigma^2[g](t)$ directly from \eqref{ul_bounds} by computing the difference between the upper and lower bounds 
	\begin{equation}\label{eq:chk_bound}
		\Delta_{\sigma}(t_0,t) = \sigma^2[g(v,t_0)] \left[( e^{2a(t-t_0)} \left( 1 + B^-_a(t_0,t) \right)^2 - e^{-2b(t-t_0)} \left(1 - B^+_b(t_0,t)\right)^2 \right],
	\end{equation}
	where $B^\pm_c(t_0,t)$ are the quantities defined in \eqref{B}. Then we can use  $\Delta_{\sigma}(t_0,t)$ to control the decay of the variance $\sigma^2[g](t)$ in order to keep the variance of $g(v,t)$  below a fixed threshold $\delta>0$.
	In this way, we can find time $t_1>t_0$ such that $\Delta_{\sigma}(t_0,t_1)>\delta$ and evolve the dynamics in the time interval $[t_0,t_1]$. The procedure is reinitialized updating the state of the linearized dynamics at time $t_1$ by setting $f(t_1,v) \equiv g(t_1,v)$.
	We formalize this procedure in Algorithm \ref{alg:MPCsigma}.
	\begin{algorithm}[!ht]
		\caption{[MdPC($\sigma^2$)]}\label{alg:MPCsigma}
		\begin{algorithmic}
			\STATE 0. Set $k \leftarrow 0$, $t_k = 0$, $g^k(v) = g(v,0)$, $f^k(v) = g(v,0)$ and tolerance $\delta$\;
			\STATE 1. Solve the Riccati equation to obtain $\kd$, $\ko$ on the time interval $[0,T]$\;
			\STATE 2. Find the time $t_{k+1}$ such that $t_{k+1} := \min \lbrace t \vert t_k< t \leq T,\Delta_{\sigma}(t_k,t) > \delta \rbrace$
			\WHILE{$t_{k+1}\leq T$}
			\STATE i. Evolve the dynamics \eqref{discreteMF_FFA} up to $t_{k+1}$\;
			\STATE ii. Set $g_{k+1}(v) = g(v,t_{k+1})$,  $f_{k+1}(v) = g(v,t_{k+1})$\;
			\STATE iii. $k \leftarrow k+1$
			\STATE iv. Compute $t_{k+1}$ from step 2.\;
			\ENDWHILE
		\end{algorithmic}
	\end{algorithm}
	\paragraph{Mean and variance driven predictive control, $\textrm{MdPC}(m_1,\sigma^2)$} For the inexact open-loop control approach \eqref{eq:mean-field_nonlinear_0} it is necessary to modify the previous algorithm  controlling also the decay of the first moment of $g=g(v,t)$. For this, we introduce the semidiscretized mean-field model of \eqref{eq:mean-field_nonlinear_0} 
	\begin{equation}\label{discreteMF_FFA_0} 
		\begin{aligned}
			\lambda^{n+1}(v,w)    &= \Phi_{\Delta t}[\lambda^n;u^n[g^0]](v,w), \qquad \lambda^0(v,w)=\lambda(v,w,0),\quad n\geq 0\,,
		\end{aligned}
	\end{equation}
	where the control is given by 
	\begin{equation}\label{eq:OCctrl_0_disc}
		u^n[g^0](w,t_n) = \frac{1}{\nu} \left(\kd(t_n) w +\ko(t_n) m_1[g^0]\right).
	\end{equation}
    According to the bounds  \eqref{ul_bounds_0}, the decay of the variance is controlled by $\Delta_\sigma(t_0,t)$ as in \eqref{eq:chk_bound} with
	\begin{equation}\label{eq:chk_bound_0}
		B^\pm_c(t_0,t) = \frac{1}{\nu}\int_{t_0}^{t-t_0}\kd(s)e^{\pm c (s-t_0)} ds.
	\end{equation}
	However, in this case the convergence towards the desired state is not guaranteed since the decay of the first moment \eqref{eq:ave_0} does not match the moment of the linearized model \eqref{eq:mean-field_Linear}. To guarantee consensus convergence we require \eqref{eq:ave_0} to be contractive. For this, we introduce the control quantity
	\begin{equation}\label{eq:chk_ave}
		\Delta_{m}(t_0,t) = \left|1-\frac{1}{\nu}\int_{t_0}^{t-t_0}(k_d(s)+k_o(s))\ ds\right|,\qquad 0<\tau\leq 1\,,
	\end{equation}
	which we use to determine updates in Algorithm \ref{alg:MPCm_sigma}.
	\begin{algorithm}[!ht]
		\caption{[MdPC$( m_1,\sigma^2)$]}\label{alg:MPCm_sigma}
		\begin{algorithmic}
			\STATE 0. Set $k \leftarrow 0$, $t_k = 0$, $g^k(v) = g(v,0)$ and tolerances $\delta,\tau>0$\;
			\STATE 1. Solve the Riccati equation to obtain $\kd$, $\ko$ on the time interval $[0,T]$.\;
			\STATE 2. Find the times $t_{\delta},t_{\tau}$ such
			\begin{equation}\label{P2}
				\begin{aligned}
					t_{\delta } &:= \min \lbrace t \vert t_k< t \leq T,\Delta_{\sigma}(t_k,t) > \delta \rbrace,\cr
					t_{\tau}      &:= \min \lbrace t \vert t_k< t \leq T,\Delta_{m}(t_k,t) > \tau \rbrace,\cr
					t_{k+1}&:=\min\left\{t_{\delta},t_{\tau}\right\}
				\end{aligned}
			\end{equation}
			\WHILE{$t_{k+1}\leq T$}
			\STATE i. Evolve the dynamics \eqref{discreteMF_FFA_0} up to $t_{k+1}$\;
			\STATE ii. Set $g^{k+1}(v) = g(v,t_{k+1})$ and  $g^{0}(v) = g(v,t_{k+1})$ \;
			
			\STATE iii. $k \leftarrow k+1$
			\STATE iv. Compute $t_{k+1}$ from \eqref{P2};
			\ENDWHILE
		\end{algorithmic}
	\end{algorithm}

	\begin{remark}
We observe that when an update is performed at each time step, that is for values of $\delta$ small enough, the MdPC approaches is equivalent to a discretization of the closed-loop control \eqref{eq:MF_FBA}. Indeed, since for every $n\geq0,$ we have $f^n\equiv g^n$ , the mean-field model \eqref{discreteMF_FFA} reduces to
		\begin{equation}\label{discreteMF_FBA} 
			\begin{aligned}
				g^{n+1}(v)    &= \Phi_{\Delta t}[g^n;u[g^n]](v),\qquad  n\geq 0,\qquad g^0(v)=g(v,t_0),
			\end{aligned}
		\end{equation}
		where the control is given by $u[g^n](v,t_n) = \frac{1}{\nu} \left(\kd(t_n) v +\ko(t_n) m_1[g^n](t_n)\right)$ .		
	\end{remark}

	\section{Numerical Experiments}\label{sec:Numerics}
	In this section we present different numerical tests on microscopic and mean-field dynamics. We analyze three different cases: a first-order opinion dynamics, a second-order alignment model, and first-order aggregation model. 
	For the numerical solution of the mean-field model \eqref{eq:mean-field_OL} we employ mean-field Monte-Carlo methods (MFMCs) developed in \cite{albi2013binary}. These methods fall in the class of  fast algorithms developed for interacting particle systems such as direct simulation Monte-Carlo methods (DSMCs) \cite{bobylev2000theory,dimarco2010direct,babovsky1986simulation}, or most recently Random Batch Methods (RBMs) \cite{jin2020random}. 
	
	We consider $N_s$ particles $v^0\equiv\left\{v_i^0\right\}_i$ sampled from the initial distribution $g^0(v)$, and 
	we duplicate the sample defining $w^0\equiv v^0$ for the linearized dynamics. We introduce the following approximation for the mean-field dynamics
	\begin{subequations}\label{eq:MFCM}
		\begin{align}\label{eq:MFMC_FFA}
			v_i^{n+1} &= (1-\Delta t \hat P_i^n) v_i^n + \Delta t \hat P_i^n \hat V^n_i - \Delta t u_i^n,  \\
			w_i^{n+1} &= (1-\Delta t \bar p) w_i^n + \Delta t \bar p \hat m_1^n - \Delta t u_i^n,
		\end{align}
	\end{subequations}
	for $n\geq 0$ and where the quantities $\hat P_i^n$ and $\hat V^n_i$ are computed from a sub-sample of $M$ particles randomly selected from the whole ensemble of $N_s$ particles as follows
	\[
	\hat P^n_i  = \frac{1}{M}\sum_{k=1}^M P(v^n_i,v^n_{i_k}),\qquad \hat V^n_i = \frac{1}{M} \sum_{k=1}^M\frac{P(v^n_i,v^n_{i_k})}{\hat P^n_i} v^n_{i_k}, \qquad i = 1,\ldots,N_s.
	\]
	For the open-loop mean-field dynamics \eqref{discreteMF_FFA} the control $u^n_i$ is defined as
	\begin{equation}\label{eq:OLctrl}
		u_i^n = -\frac{1}{\nu}\left(\kd^nw_i^n+\ko^n \hat m_1^n \right),\qquad  \hat m_1^n = \frac{1}{N_s}\sum_{j=1}^{N_s}w_i^n.
	\end{equation}
	The scheme \eqref{eq:MFCM} reduces to a set of equations for the  mean-field dynamics for the closed-loop \eqref{discreteMF_FBA} and inexact open-loop \eqref{discreteMF_FFA_0}, respectively.
	In the closed-loop setting \eqref{discreteMF_FBA} the control term is computed as
	\begin{equation}\label{eq:CLctrl}
		u_i^n = -\frac{1}{\nu}\left(\kd^nv_i^n+\ko^n \hat m_1^n \right),\qquad  \hat m_1^n = \frac{1}{N_s}\sum_{j=1}^{N_s}v_i^n,
	\end{equation}
	and in the inexact open-loop approach \eqref{discreteMF_FFA_0} we have
	\begin{equation}\label{eq:iOLctrl}
		u_i^n = -\frac{1}{\nu}\left(\kd^nv_i^0+\ko^n \hat m_1^0 \right),\qquad  \hat m_1^0 = \frac{1}{N_s}\sum_{j=1}^{N_s}v_i^0.
	\end{equation}

	We report in Table \ref{tab:parameters} the different choices of parameters used for the numerical discretization of the mean-field dynamics and for each control approach, respectively. We compare the performance of the control laws through the discretized cost
	\begin{equation}\label{eq:runcost_ex}
		J_{\Delta t,N_s}(u,g^0) := \frac{\Delta t}{N_s}\sum_{n=0}^{N_T}\sum_{j=1}^{N_s} (\vert v^n_j \vert^2 + \nu \vert u_j^n\vert^2), 
	\end{equation}
	with time step $\Delta t$ and $N_s$ Monte Carlo samples.

	\begin{table}
		\centering
		\begin{tabular}{cccccccc}		
			& $N_s$ & M & $\Delta t$ & $\nu$ &  $T$ & $\delta$ & $\tau$\\
 \cmidrule(lr){2-2}\cmidrule(lr){3-3}\cmidrule(lr){4-4}\cmidrule(lr){5-5}\cmidrule(lr){6-6}\cmidrule(lr){7-7}\cmidrule(lr){8-8}\\[-1em]
			Test 1: Opinion formation & 1e4 & 100 & 1e-2  & 1e-2 & $1$ & 1e-1 &1\\
			Test 2: Cucker-Smale dynamics & 1e5 & 100 & 5e-2  & 1e-1 & $3$ & 1 &--\\
			Test 3: Aggregation dynamics & 1e5 & 10 & 1e-2  & 1 & $7$ & 1e-1 &--\\
			\hline\\
		\end{tabular}
		\caption {Simulation and optimization parameters for each test case.}\label{tab:parameters}
		\vskip-5mm
	\end{table}

	\subsection{Test 1: Opinion formation} \label{test1section}
	We show an example in the context of opinion formation by Hegselmann and Krause \cite{hegselmann2002opinion}. We consider the positive interaction kernel defined as 
	$P(v,w)=C \cdot \chi (\vert w-v \vert < \eta) $, where $\eta=0.25$ represents the confidence level and with a constant $C= 10$, and $\chi$ is an indicator function.  The initial density of particle $g^0(v)$ is chosen such that consensus towards the target $\bar v  = 0$  would not be reached without control action, e.g. $g^0(v)=\frac23\chi_{[1/4,7/4]}(v)$ . 
	We use the forward scheme \eqref{eq:MFCM} with fixed time step $\Delta t = 0.01$, sampling size $N_s =10000$ of the initial distribution $g^0(v)$ and fixed $M=100$ for the approximation of the non-local interactions.
	To compare the mean-field dynamics with the microscopic we simulate $N=50$ agents uniformly sampled from $g^0$. 
	
	In the top row of Figure \ref{test1} the uncontrolled dynamics are shown, where clusters of opinions emerge due to  structure of the interaction kernel $P$.
	The second and third rows of Figure \ref{test1} depict the controlled dynamics for the microscopic and the mean-field  dynamics.
	The left column of Figure \ref{test1} illustrates the convergence to the target when the MdPC$(m_1,\sigma^2)$ is applied. Algorithm \ref{alg:MPCm_sigma} is used with $\delta$ and $\tau$ chosen according to Table \eqref{tab:parameters}.
	The vertical lines in the plots represent the times of the update. 
	We have a different situation with the MdPC$(\sigma^2)$ approach \ref{alg:MPCsigma}, depicted in the middle column of Figure \ref{test1}. In this case the control is applied by directly embedding the linear synthesis into the the non-linear dynamics. As a consequence, we also require the evolution of the linear state that we plot in a dashed green line for the microscopic case. The right column of Figure \ref{test1} reports the closed-loop control results.
	\begin{figure}[!ht]
		\begin{center}
			{Uncontrolled} \\
			\includegraphics[width=0.328\linewidth]{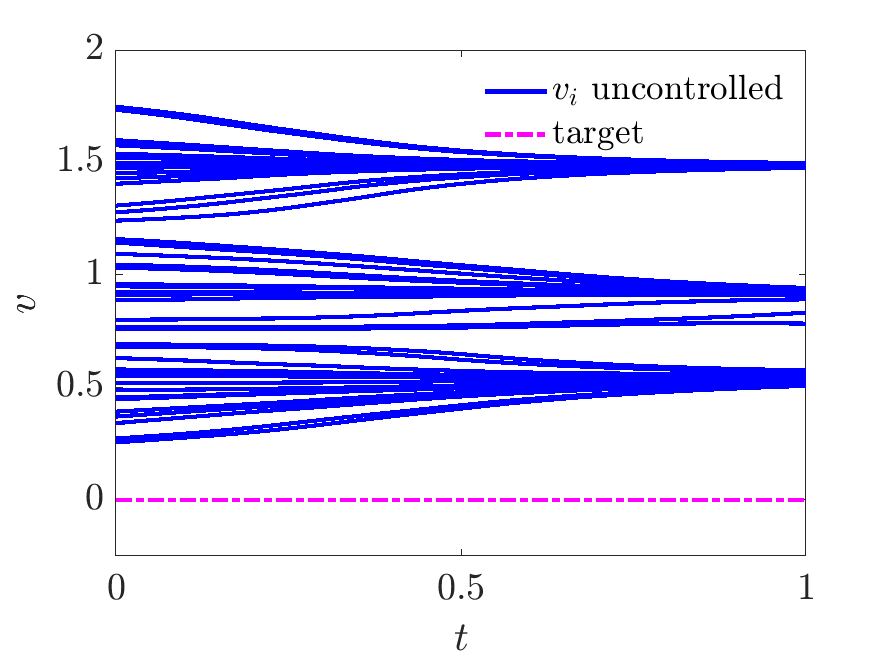}
			\includegraphics[width=0.328\linewidth]{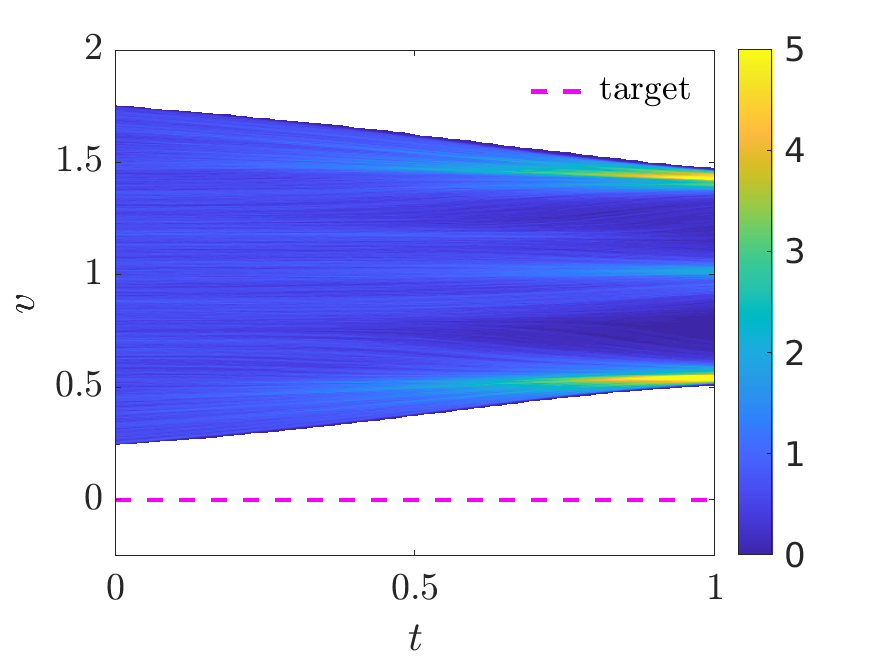}
			\\
			\
			\\
			{MdPC($m_1,\sigma^2$)}\hspace{2.5cm}{MdPC($\sigma^2$)}\hspace{2.5cm}{closed-loop} \\
			\includegraphics[width=0.328\linewidth]{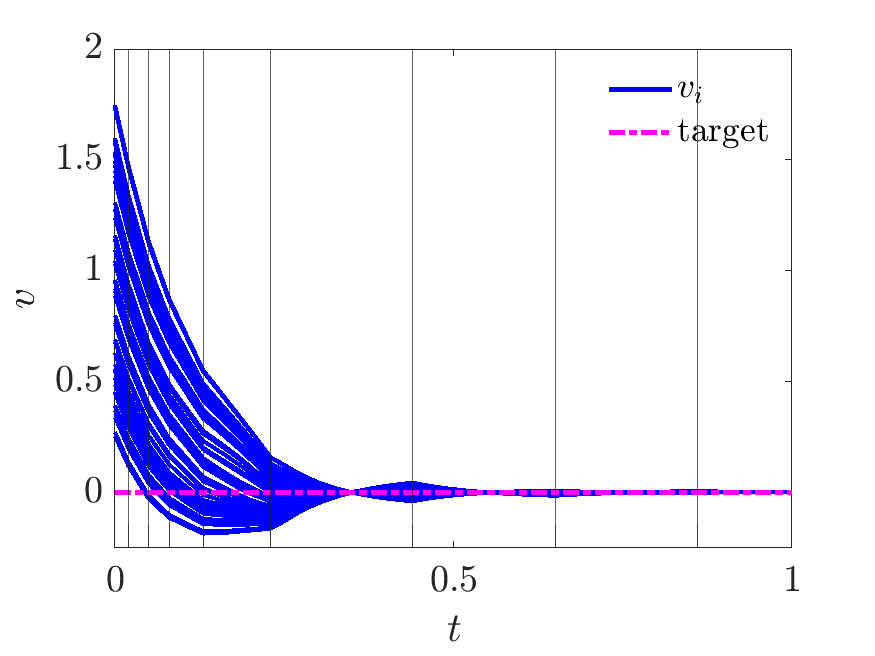}
			\includegraphics[width=0.328\linewidth]{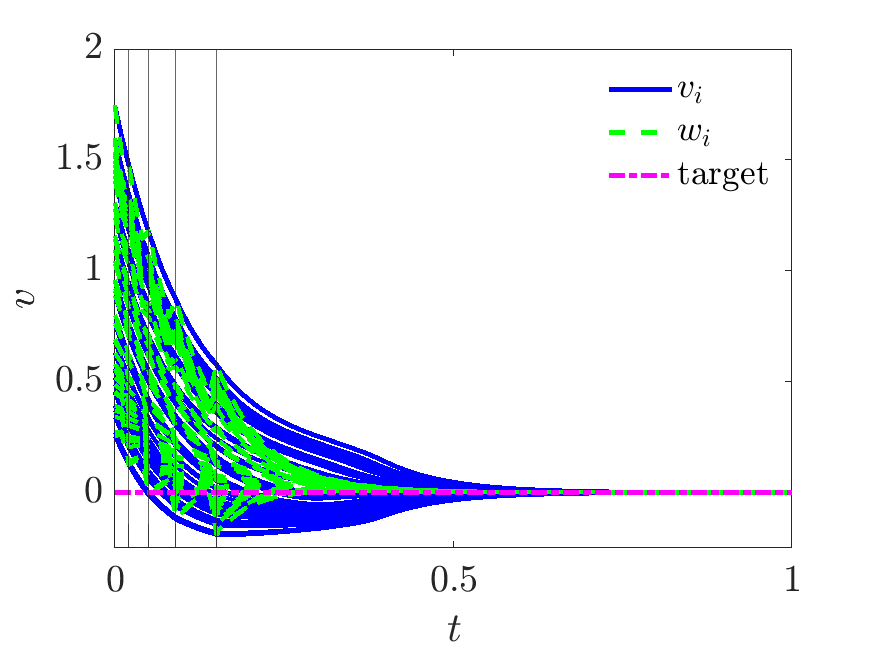}
			\includegraphics[width=0.328\linewidth]{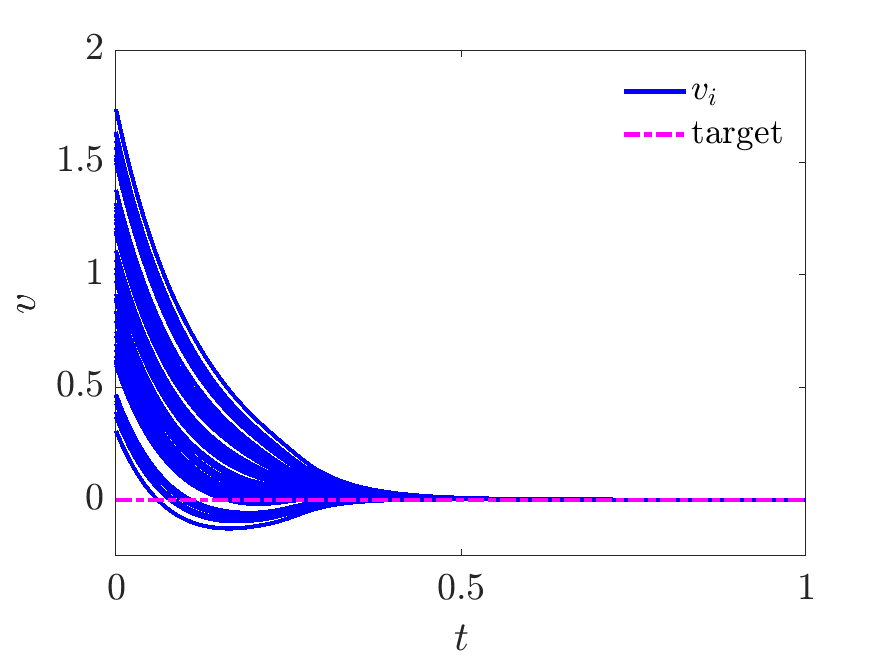}
			\\
			\
			\\
			\includegraphics[width=0.328\linewidth]{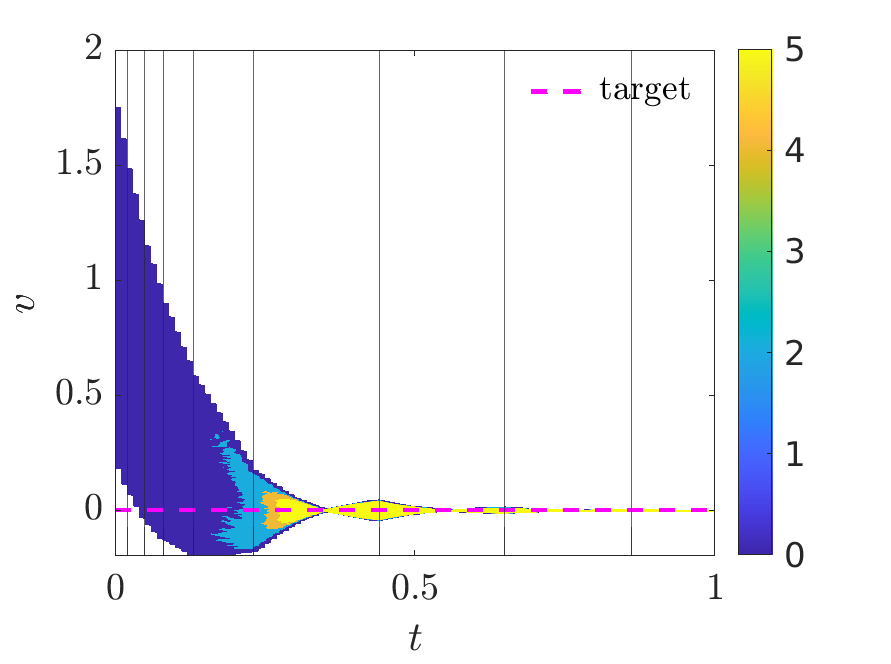}
			\includegraphics[width=0.328\linewidth]{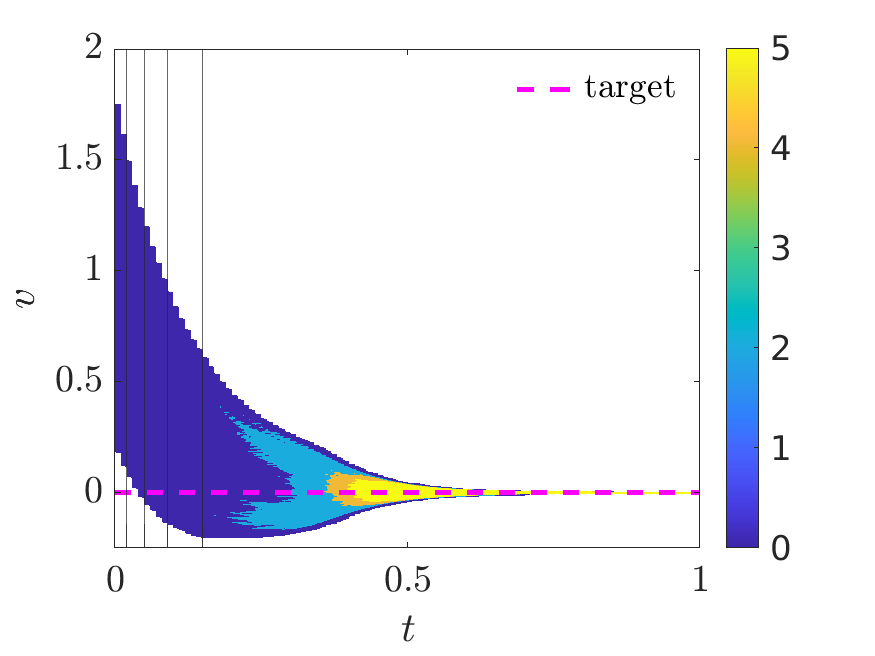}
			\includegraphics[width=0.327\linewidth]{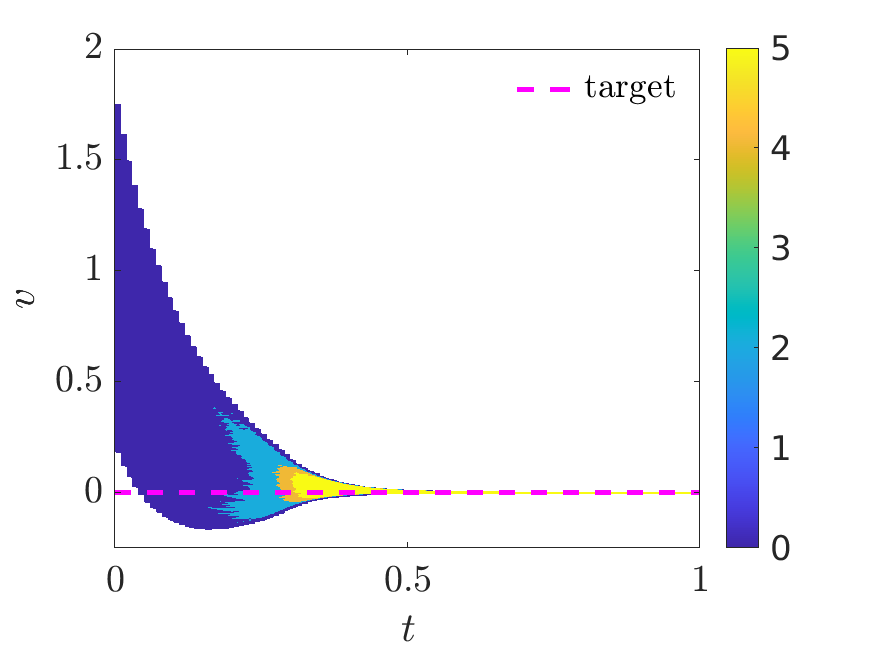}
			\caption{{\em Test 1.} {Top: uncontrolled discrete and mean-field evolution of the model. Middle: the controlled discrete case. Bottom: the controlled mean-field dynamics.}} \label{test1}
		\end{center}
	\end{figure}
	To better interpret these results, we perform a numerical analysis to study the decay of the variance $\sigma^2 [g]$ using different values of $\delta$ in Algorithm \ref{alg:MPCsigma} and Algorithm \ref{alg:MPCm_sigma}. Figure \ref{test1bound} compares the variance of the system as the values of the tolerance $\delta$ changes. It can be seen that as $\delta$ decreases, the MdPC($\sigma^2$) approaches the closed-loop control. This numerical evidence  is further confirmed by Table \ref{test1table1}. With  decreasing values of $\delta$ we have an increasing number of updates and the values of the functional $J_{\Delta t,N_s}$ computed in \eqref{eq:runcost_ex} is similar for the three control approaches. We observe that the closed-loop control corresponds to a limit case of the moment-driven MPC methods.
	\begin{figure}[!ht]
		\begin{center}
			\includegraphics[width=0.495\linewidth]{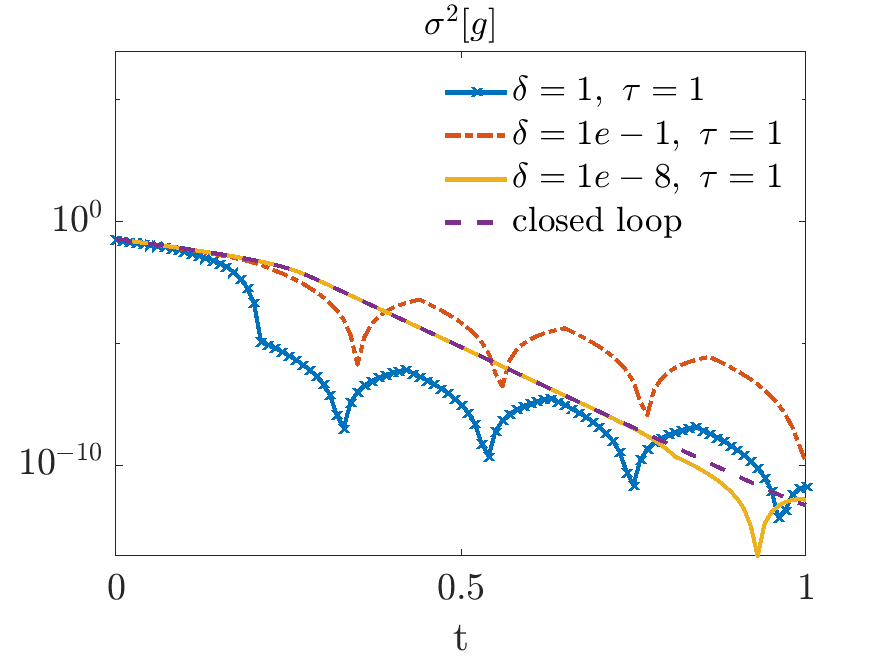}
			\includegraphics[width=0.495\linewidth]{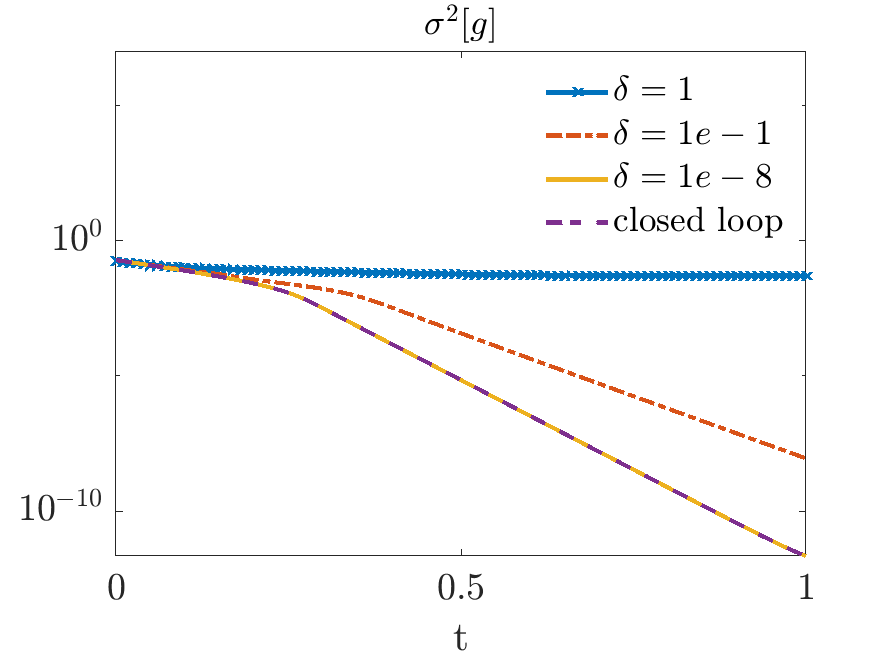}
			\caption{{\em Test 1.} {Semi-log plot with a comparison between variance decay with different values of the tolerance $\delta$. Left: variance decays for MdPC$(m_1,\sigma^2)$. Right: decays for MdPC($\sigma^2$)} compared with respect to the closed-loop control.}\label{test1bound}
		\end{center}
	\end{figure}
	\begin{table}[!ht]
			\begin{tabular}{ccccccc}
				\multicolumn{1}{c}{} & \multicolumn{3}{c}{MdPC($m_1,\sigma^2$)} 
				& \multicolumn{3}{c}{MdPC($\sigma^2$)}\\
				\cmidrule(lr){2-4} \cmidrule(lr){5-7}\\[-1em]
				$\delta$ & 1 & $0.1$ & 1e-8  &  1 & 0.1 & 1e-8 \\
				\texttt{ update} ($\%$) & 4 \% & 8\% & 71 \%   
				& 0 \%      & 4 \%     & 72 \%   
				\\
				$\sigma^2[g](T)$       & 1.28e-11 & 1.26e-10 & 3.80e-12  
				& 5.04e-2 & 8.94e-9 & 2.22e-12   
				\\
				$J_{\Delta t,N_s}$  &   1.8131 & 0.1306 & 0.1281
				& 0.1777 & 0.1309 & 0.1281  
				\\   
				\hline
			\end{tabular}\\
		\caption{{\em Test 1.} We compare the different MdPC approaches with respect to the closed-loop control. For MdPC($m_1,\sigma^2$) the tolerance for the mean is set $\tau=1$. The number of updates indicates the percentage of control updates over the total number of time steps $N_T = 100$ (for reference, the closed-loop control would require a 100\%). The final value of the variance is denoted by $\sigma^2(T)$, and $J_{\Delta t,N_s}$ corresponds to the value of the cost functional \eqref{eq:runcost_ex}. For the closed-loop control $\sigma^2(T)=3.80e-12$ and $J_{\Delta t,N_s}=0.1281 $. }\label{test1table1}
		\vskip-5mm
	\end{table}
	\subsection{Test 2: Cucker-Smale dynamics}
	We study alignment in a second-order, 1D model with Cucker-Smale type interactions \cite{cucker2007emergent}. We consider a state  characterized by  $(x_i,v_i)\in\R^{2} $. The interaction kernel is given by $P(x,y)= \frac{1}{(1+\vert y-x \vert^2)^{\gamma}},$ with $\gamma\geq 0,$  which is a decreasing function of the relative distance, bounded to $[0,1]$. Under the condition $\gamma\geq 1/2$, the convergence to consensus of the free dynamics depends on the initial state \cite{cucker2007emergent}. We set $\gamma=2$ and a suitable initial state, such that the flocking state is not achieved without control action. To perform our analysis we refer to the second-order dynamics \eqref{order2} and use Remark \ref{rmk_second order} to obtain the Riccati equations. Unlike the first-order dynamics, the constrained mean-field has a transport term
	\begin{equation}\label{eq:MF_CS}
		\partial_t g + v\cdot\nabla_x g = - \nabla_v \cdot\left( g\left(\mathcal{P}[g] + u(t) \right) \right), \qquad\qquad\,\, g(x,v,0)=g^0(x,v),
	\end{equation}
	and the nonlocal operator $\mathcal P $ is defined as
	\[\mathcal{P}[g](x,v,t) = \int_{\R^d\times\R^d} P(x,y)(w-v)g(y,w,t)\ dy\ dw.\]
	We discretize the mean-field model employing the forward scheme \eqref{eq:MFCM} with fixed time step $\Delta t = 0.05$, sampling size of $N_s =100000$ particles and fixed $M=100$. In order to treat the additional transport term in the dynamics, we use a splitting method to perform the free transport step. The control is computed by means of MdPC($\sigma^2$). We refer to Table \eqref{tab:parameters} for the choice of parameters.
	
	Figure \ref{test2} presents the initial data at the top, which is a bivariate distribution unimodal in space and bimodal in velocity defined as follows:
	\[
	g^0(x,v) = \frac{1}{4 \pi \sigma_x \sigma_v} \text{exp}\Biggl( -\frac{x^2}{2 \sigma_x^2}\Biggr) \Biggl[ \text{exp}\Biggl( -\frac{(v+v_-)^2}{2 \sigma_v^2}\Biggr) + \text{exp}\Biggl( -\frac{(v+v_+)^2}{2 \sigma_v^2}\Biggr) \Biggr],
	\]
	with $\sigma_x=0.2, \sigma_v=0.4$ and $v_{-} =-1, v_+ =4$.
	We compare the evolution of the density distribution in the phase space $(x,v)$, jointly with a set of $N=30$ microscopic points $(x_i(t),v_i(t))$ sampled from the initial distribution.
	The middle row depicts two time frames of the uncontrolled dynamics, where alignment is not reached. In the bottom row we report the constrained dynamics, where
	the alignment is reached at time $T = 3$ and the density $g(x,v,t)$  concentrates at the target state $\bar v = 0$, whereas its support is bounded in space.
	\begin{figure}[!ht]
		\begin{center}
			\includegraphics[width=0.45\linewidth]{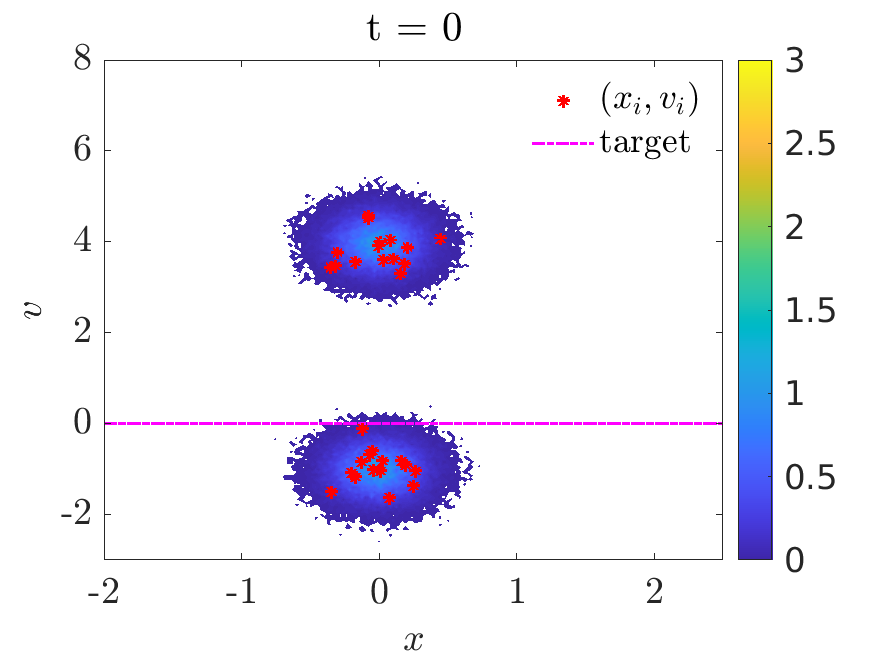}
			\\
			\textrm{Uncontrolled} \\
			\includegraphics[width=0.45\linewidth]{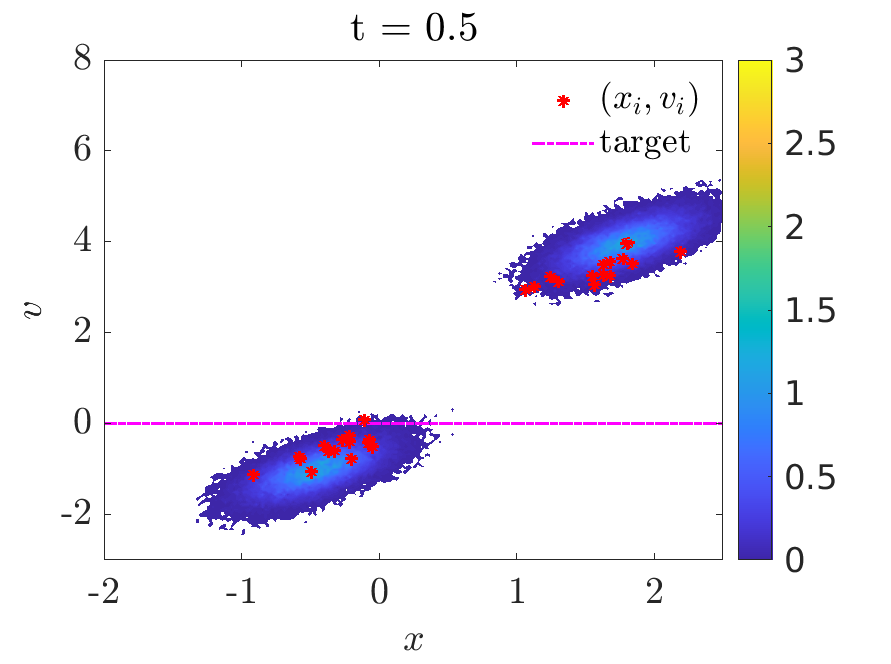}
			\includegraphics[width=0.45\linewidth]{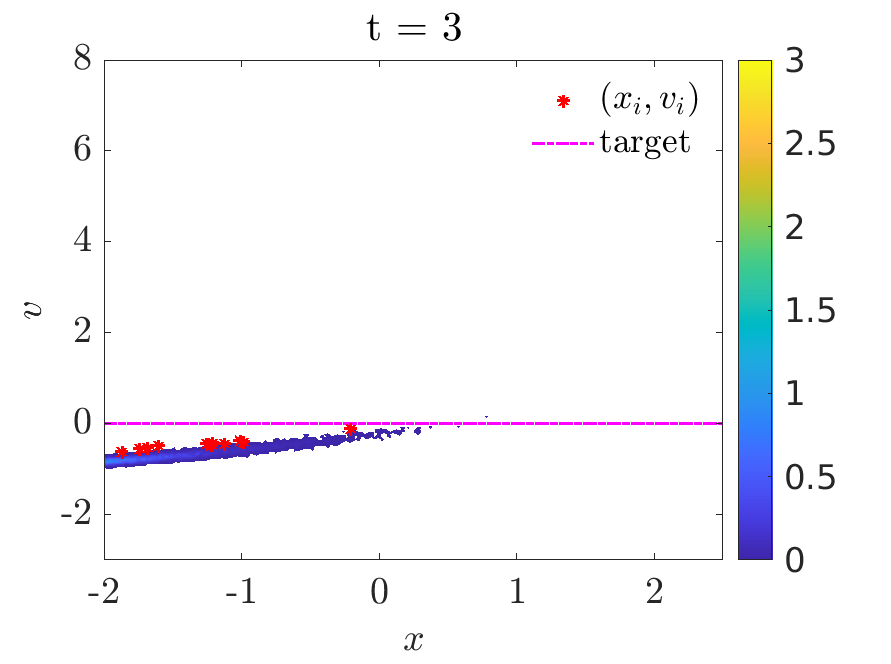}
			\\
			\textrm{$\textrm{MdPC}(\sigma^2)$} \\
			\includegraphics[width=0.45\linewidth]{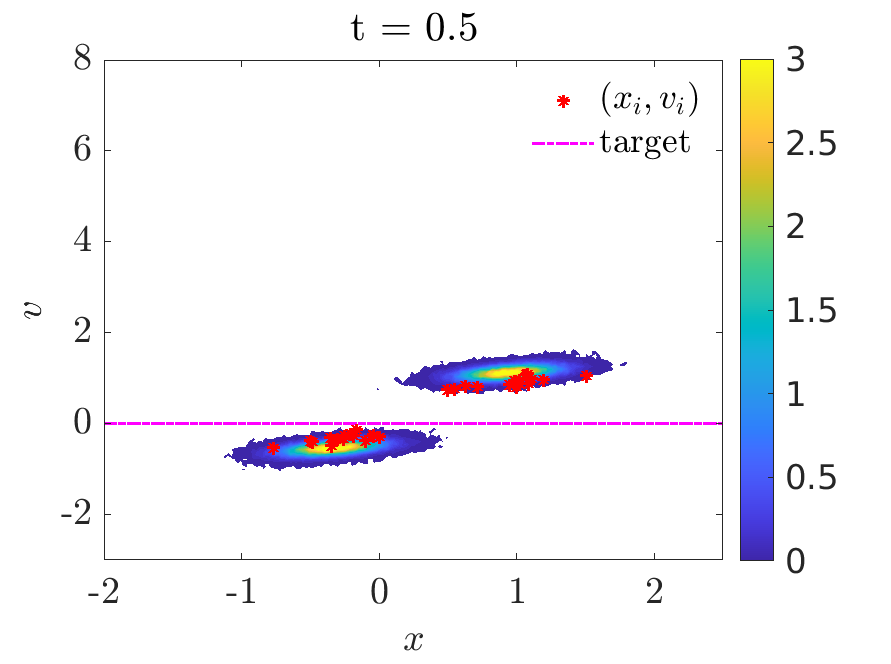}
			\includegraphics[width=0.45\linewidth]{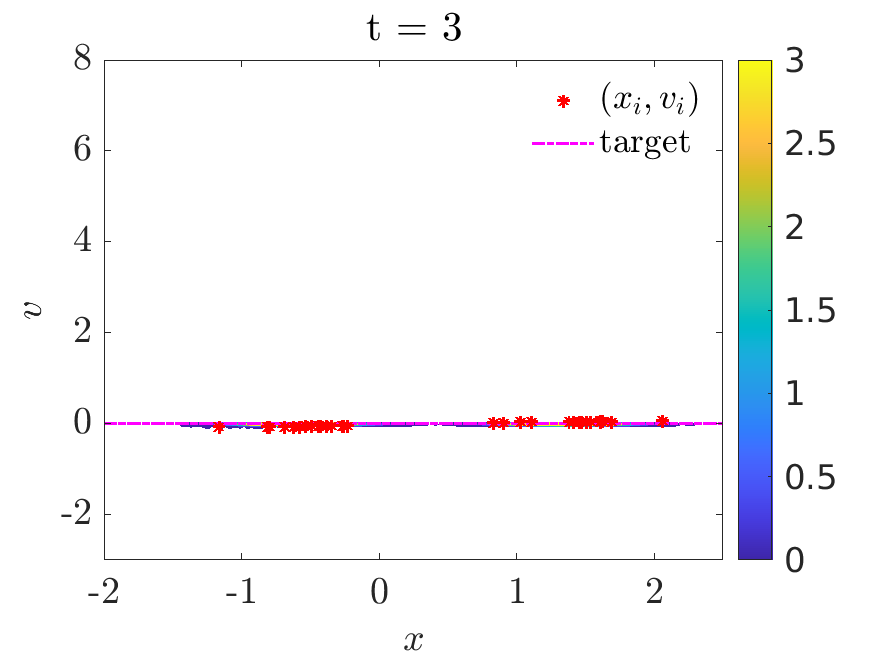}
		\end{center}
		\caption{{\em Test 2.}{ Comparison between the uncontrolled and controlled mean-field evolution of the Cucker-Smale dynamics with $\textrm{MdPC}(\sigma^2)$ and $\delta = 1$. Without control intervention the alignment state is not reached.}}\label{test2}
	\end{figure}

	\begin{figure}[!ht]
		\begin{center}
			\includegraphics[width=0.45\linewidth]{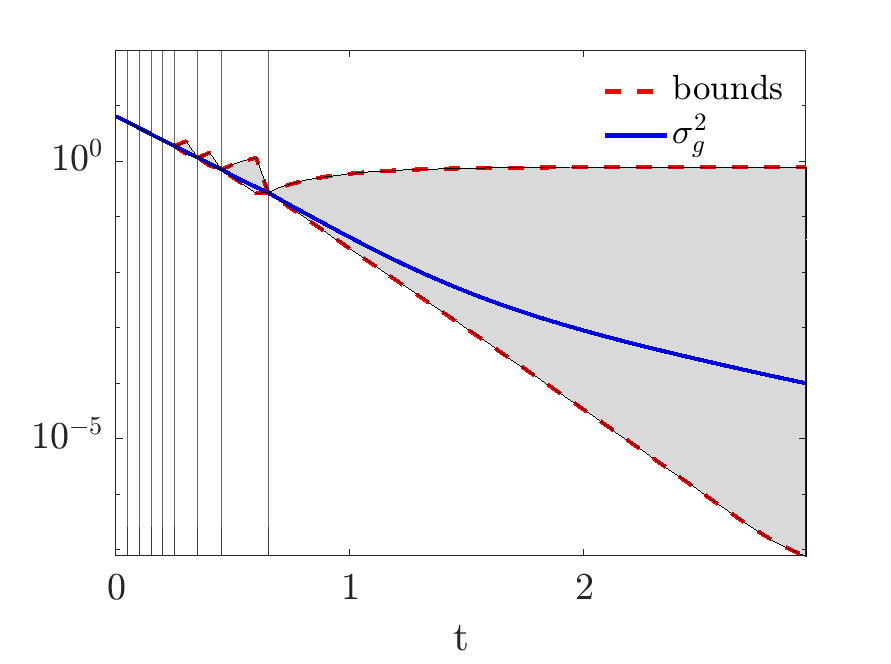}
			\includegraphics[width=0.45\linewidth]{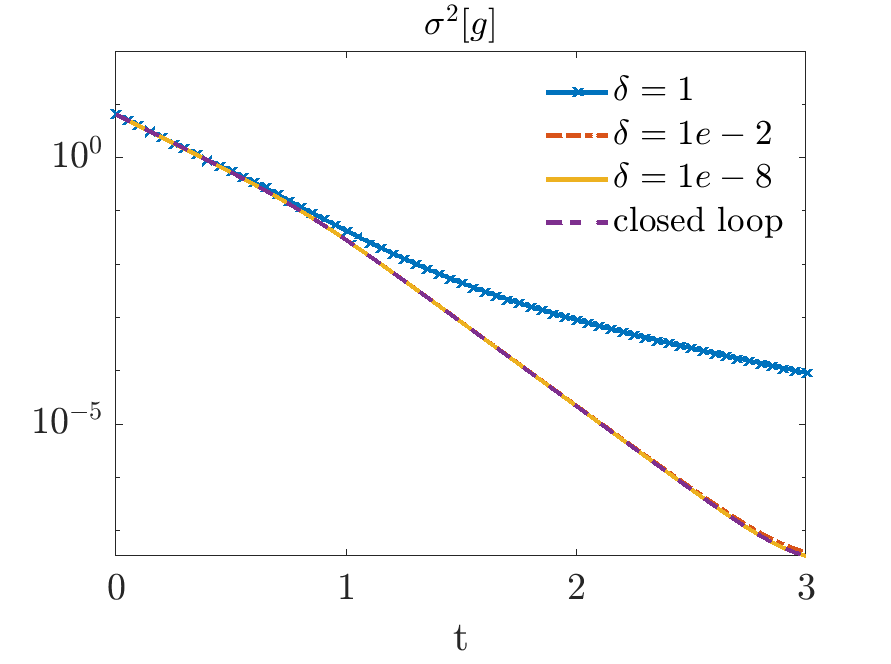}
			\caption{{\em Test 2.}{Left: variance and bounds for the second order attractive one-dimensional mean-field dynamics with tolerance $\delta = 1$. Right: variance decays with different values of $\delta$}.}\label{test2bound}
		\end{center}
	\end{figure}
	We perform a numerical study for the variance decay $\sigma^2[g]$ using different tolerances $\delta$ in Algorithm \ref{alg:MPCsigma} for $\textrm{MdPC}(\sigma^2)$. Figure \ref{test2bound} compares the decay of the variance of the system as the values of the tolerance $\delta$ changes. On the left we observe the decay for $\delta =1$ jointly with the 
	update and the evolution of the variance bounds. On the right, it can be seen that as $\delta$ decreases $\textrm{MdPC}(\sigma^2)$ is similar to the closed-loop control dynamics. 
	Table \ref{test2table} quantifies the performances of the MdPC($\sigma^2$) reporting the percentage of control updates performed over $N_T = 60$ steps, the variance $\sigma^2[g]$ at time $T=3$, and the value  of the cost functional \eqref{eq:runcost_ex}.

	\begin{table}[!ht]
		\center
		{
			\begin{tabular}{ccccc}
				& \multicolumn{3}{c}{MdPC($\sigma^2$)} 
				& \multicolumn{1}{c}{closed-loop}\\
 				\cmidrule(lr){2-4} \cmidrule(lr){5-5}\\[-1em]
				$\delta$  & 1 &1e-2 & 1e-8  & - - \\
				{ update} ($\%$) & 13 \%      & 40 \%     & 99 \%   & 100\%   \\
				$\sigma^2[g](T)$       & 9.1393e-05 & 3.9247e-08 & 3.3374e-08  & 3.3371e-08  \\    
				$J_{\Delta t,N_s}$       & 3.0059
				& 2.9976 & 2.9951   & 2.9951 \\   
				\hline\\
			\end{tabular}
		}
		\caption {{\em Test 2.} Number of updates and final values of the variance using different values of $\delta$.  We compare the different control approaches. 
		The update percentage is computed over $N_T = 60$.}\label{test2table}
	\vskip-5mm
	\end{table}
	\subsection{Test 3: Aggregation dynamics}
	The last example is a first-order aggregation model in 2D, where agents interact according to an attraction-repulsion kernel. We consider the following interaction kernel $P(v,w) = \vert w-v \vert^{\alpha - 2} - \vert w-v \vert^{\beta - 2},$ where $\alpha = 4$ and $\beta = 2$.
	For these specific values of the parameters it can be shown that the equilibrium configuration is an uniform distribution on an annulus of radius $R = \frac{1}{\sqrt{3}}$, and same center of mass as the initial distribution. For analytical and numerical characterizations of the equilbrium of these models we refer to \cite{BCLR13}. We consider an initial density of particles uniformly distributed on the 2D disc of radius $R_0 = \frac{2}{\sqrt{3}}$ centered in $(-1,1)$, that is
	\begin{equation}
			g^0(v) = \frac{1}{\vert \mathcal{C} \vert} \chi_{\mathcal{C}}(v)\,,\qquad\mathcal{C}:= \lbrace v \in \mathbb{R}^2 : \vert v - (-1,1)^\top \vert \leq R_0 \rbrace\,,
	\end{equation}
and $|\mathcal{C}|$ denoting its volume.  In order to simulate the dynamics we consider $N_s = 10^5$ particles sampled from $g^0(v)$ and we implement the forward scheme \eqref{eq:MFCM} with fixed time step $\Delta t = 0.01$. We select $M=10$ particles for the approximation of the non-local interactions.  
	We use the MdPC($\sigma^2$) approach with a penalization factor $\nu = 1$ and a stopping tolerance $\delta = 0.1$.
	In Figure \ref{test3} we report the evolution of the mean-field and the microscopic dynamics. The latter is sampled with $N=30$ particles from $g^0$.
	The second row of Figure \ref{test1} shows the uncontrolled dynamics, where mass concentrates towards a 2D annulus of radius $R = \frac{1}{\sqrt{3}}$.
	While the third row depicts the open-loop control case where MdPC($\sigma^2$) is applied. At time $T = 7$ the distribution is converged to a concentration at $\bar{v}=(\bar{v}_1,\bar{v}_2)=(0,0)$.

	As in previous tests we illustrates a comparison between different open-loop controls varying tolerance $\delta$ in algorithms \ref{alg:MPCsigma}.
	Figure \ref{test3bound} and Table \ref{test3table} highlight that with a very small value of $\delta$, the MdPC approach coincides with the closed-loop approach.
	\begin{figure}[!ht]
		\begin{center}
			\includegraphics[width=0.45\linewidth]{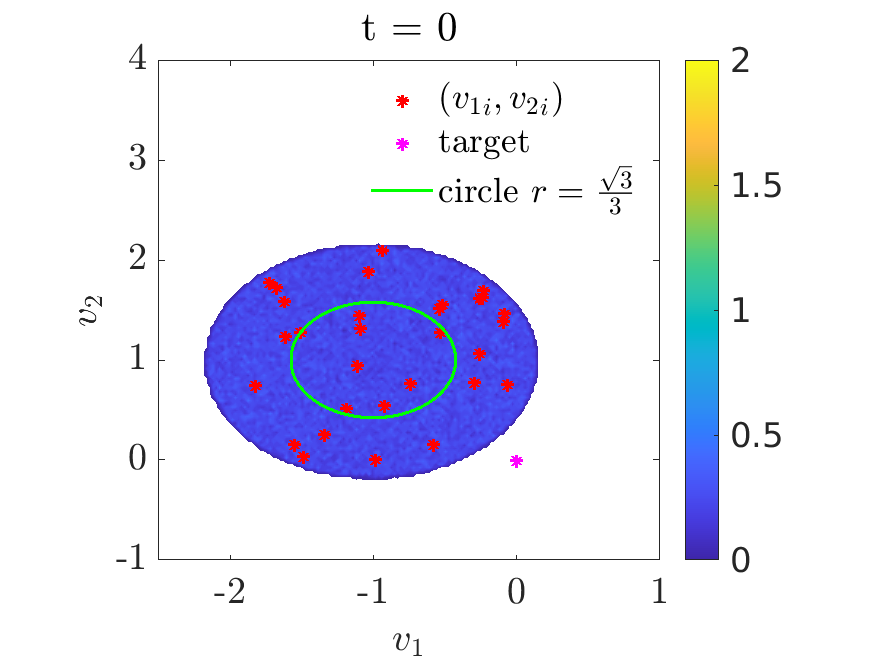}
			\\
			\textrm{Uncontrolled dynamics} \\
			\includegraphics[width=0.45\linewidth]{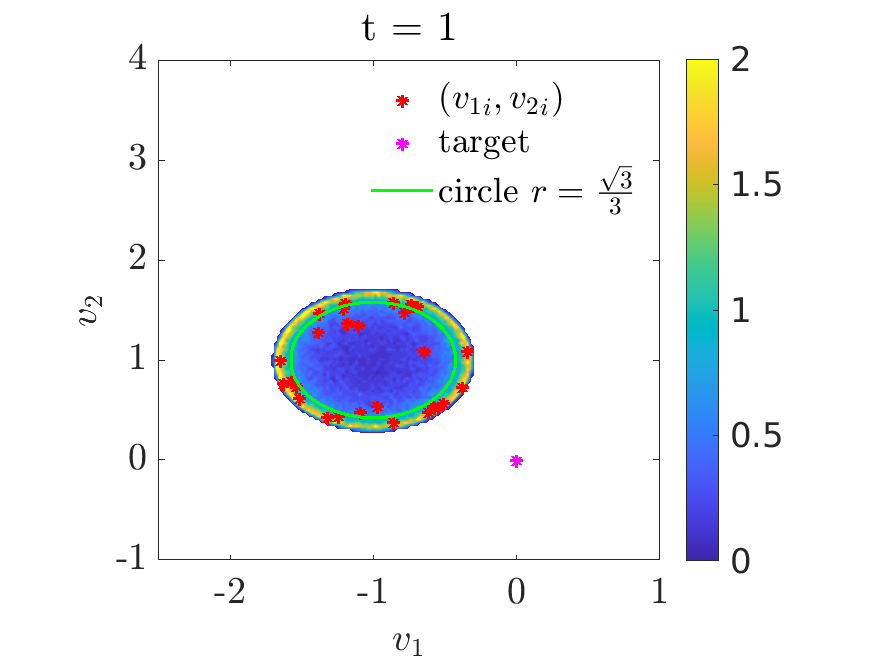}
			\includegraphics[width=0.45\linewidth]{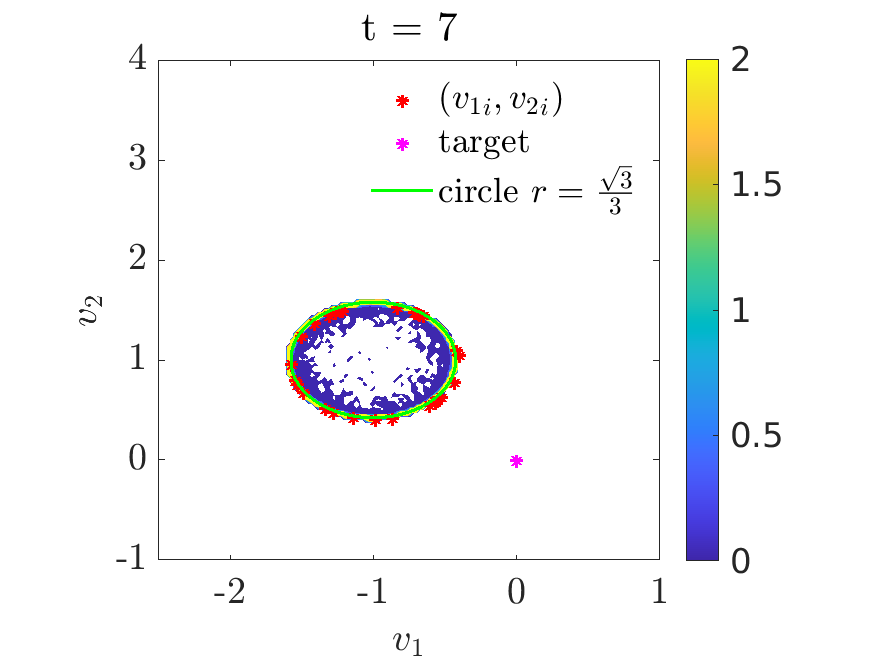}
			\\
			{$\textrm{MdPC}(\sigma^2)$} \\
			\includegraphics[width=0.45\linewidth]{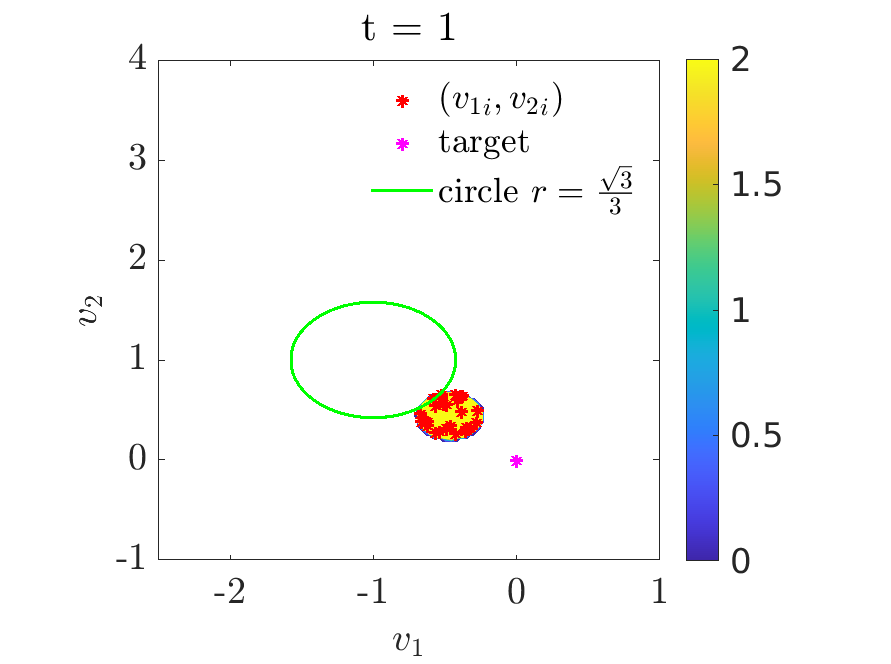}
			\includegraphics[width=0.45\linewidth]{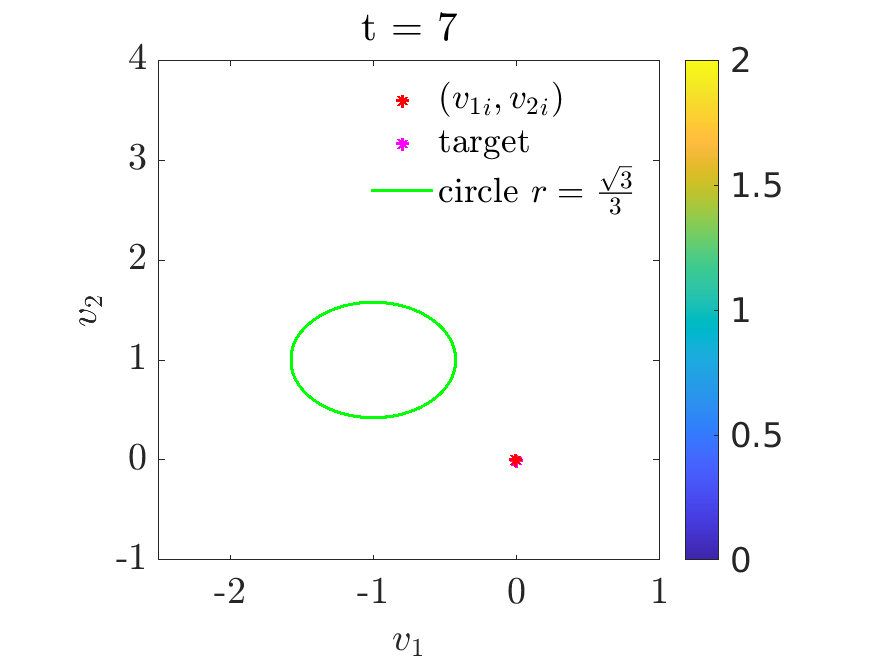}
			
		\end{center}
		\caption{{\em Test 3.} {Comparison between the uncontrolled and controlled mean-field evolution of the aggregation dynamics with $\textrm{MdPC}(\sigma^2)$ and $\delta = 0.1$. Without control intervention the alignment state is not reached.}} \label{test3}
	\end{figure}

	\begin{figure}[!ht]\label{test3bound}
		\begin{center}
			\includegraphics[width=0.45\linewidth]{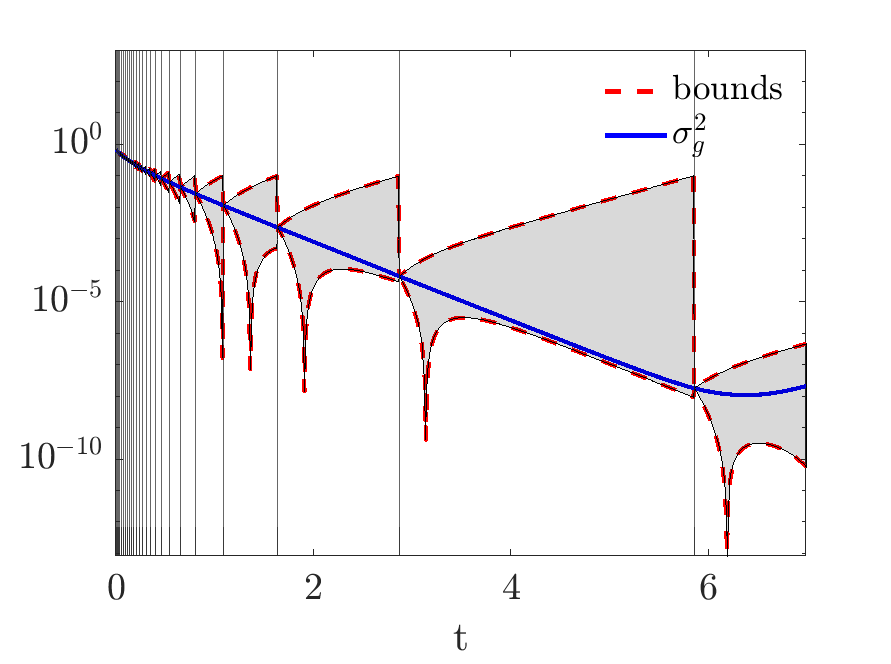}
			\includegraphics[width=0.45\linewidth]{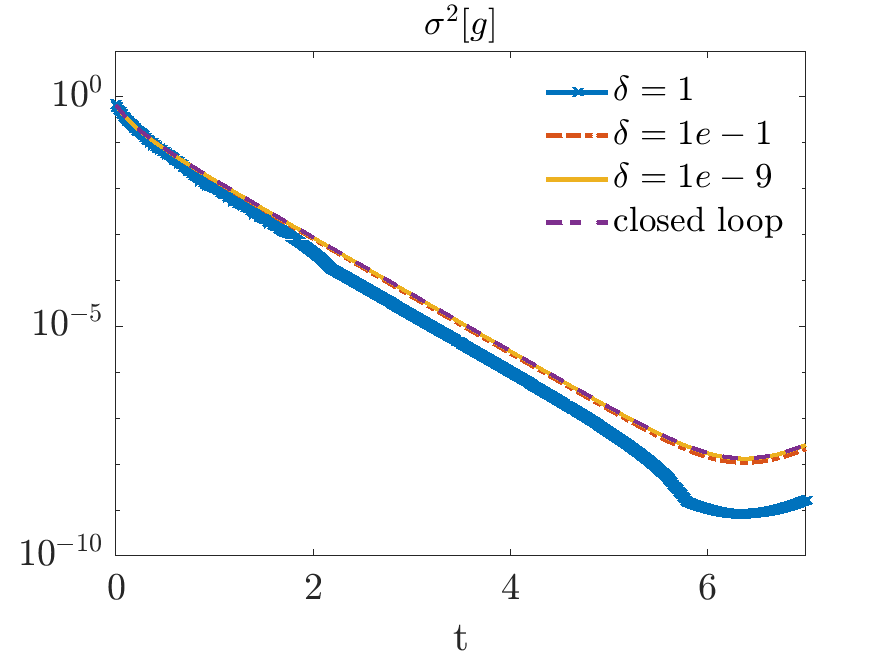}
			\caption{{\em Test 3.} {Left: variance evolution and bounds for the first order attractive-repulsive one-dimensional mean-field dynamics with tolerance $\delta = 0.1$. Right: variance decays with different values of $\delta$}.}\end{center}
	\end{figure}

	\begin{table}[!ht] \label{test3table}
		\center
		\begin{tabular}{ccccc}
			& \multicolumn{3}{c}{MdPC($\sigma^2$)} 
			& \multicolumn{1}{c}{closed-loop}\\
			\cmidrule(lr){2-4}\cmidrule(lr){5-5}\\[-1em]
			$\delta$  & 1 &1e-1 & 1e-9  & - - \\
			\texttt{ update} ($\%$)  & 1 \%      & 4 \%     & 99 \%    & 100\%   \\
			$\sigma^2[g](T)$       & 1.6875e-09 & 2.1253e-08 & 2.6151e-08  & 2.6187e-08  \\    
			$J_{\Delta t,N_s}$       & 3.0459 & 2.9751 & 2.9750   & 2.9750 \\   
			\hline\\
		\end{tabular}
		\caption {{\em Test 3: Aggregation dynamics}. Number of updates and final values of the variance using different values of $\delta$.  
			The update percentage is computed over $N_T = 700$.}
		\vskip-7mm
	\end{table}

	\paragraph{Concluding Remarks} We have studied the design of control laws for interacting particle system based on the solution of the optimal control problem associated to linearized dynamics. We have assessed the impact of different sub-optimal control laws into the original non-linear dynamics deriving mean-field limits of the microscopic constrained systems and estimating analytically and numerically the decay of the first and second moments. We proposed a novel numerical technique based on the moments decay (MdPC). In particular, we obtain a hierarchy of approximations from open-loop to closed-loop control by scaling the tolerance level. These strategies have shown to be robust even with considerably fewer updates of the control law for the non-linear dynamics. The proposed methodology expands the existing NMPC literature by developing a new paradigm in which the control laws are updated based on dynamic information of the system. Here, the use of moments information is a particular example suitable in the context of mean-field dynamics. Further extensions and analysis will include the study of other dynamic indicators for control update, in particular those that could be linked to a physical observable of the system, and the incorporation of nonlinear state estimation in the control loop.

	\bibliographystyle{abbrv}
	\bibliography{biblioCS2}
\end{document}